\documentclass[ejs]{imsart}

\RequirePackage[OT1]{fontenc}
\RequirePackage{amsthm,amsmath}
\RequirePackage[numbers]{natbib}
\RequirePackage[colorlinks,citecolor=blue,urlcolor=blue]{hyperref}\usepackage{amsfonts}
\RequirePackage{multicol}
\RequirePackage{caption}
\RequirePackage{multicol}
\RequirePackage{float}
\RequirePackage{bbold}
\RequirePackage{url}
\RequirePackage{tabularx}
\RequirePackage{grffile}
\RequirePackage[english]{babel}
\RequirePackage{enumerate}
\RequirePackage{rotating}
\pubyear{0}
\volume{0}
\issue{0}
\firstpage{1}
\lastpage{8}

\startlocaldefs
\numberwithin{equation}{section}
\theoremstyle{plain}
\newtheorem{Theorem}{Theorem}[section]
\newtheorem{Definition}{Definition}[section]
\newtheorem{Proposition}{Proposition}[section]
\newtheorem{Lemma}{Lemma}[section]
\newtheorem{Corollary}{Corollary}[section]

\newtheorem{Remark}{Remark}[section]

\newtheorem{Assumption}{Assumption}[section]

\newcommand{\N}{\mathbb{N}}
\newcommand{\Z}{\mathbb{Z}}
\newcommand{\R}{\mathbb{R}}

\newcommand{\E}{\mathbb{E}}

\endlocaldefs

\begin{document}
	
	\begin{frontmatter}
		\title{Weak dependence and GMM estimation of supOU and mixed moving average processes}
		\runtitle{Weak dependence and GMM estimation of MMA}

		\begin{aug}
			
			\author{\fnms{Imma Valentina} \snm{Curato}\ead[label=e1]{imma.curato@uni-ulm.de}}
			\and
			\author{\fnms{Robert} \snm{Stelzer}
				\ead[label=e3]{robert.stelzer@uni-ulm.de}}
			
			\address{Ulm University, Institute of Mathematical Finance,\\
				Helmholtzstra\ss e 18, 89069 Ulm, Germany.\\
				\printead{e1,e3}}

			\runauthor{I.V. Curato and R. Stelzer}

		\end{aug}

\begin{abstract}
We consider a mixed moving average (MMA) process $X$ driven by a L\'evy basis and prove that it is weakly dependent with rates computable in terms of the moving average kernel and the characteristic quadruple of the L\'evy basis. Using this property, we show conditions ensuring that sample mean and autocovariances of $X$ have a limiting normal distribution.
We extend these results to stochastic volatility models and then investigate a Generalized Method of Moments estimator for the supOU process and the supOU stochastic volatility model after choosing a suitable distribution for the mean reversion parameter. For these estimators, we analyze the asymptotic behavior in detail.
\end{abstract}



\end{frontmatter}

\section{Introduction}

L\'evy-driven continuous-time moving average processes, i.e. processes $(X_t)_{t \in \R}$ of the form $X_t=\int_{\R} f(t-s) dL_s$ with $f$ a deterministic function and $L$ a L\'evy process, are frequently used to model time series, especially, when dealing with data observed at high frequency. Moreover, causal moving averages can be used to model the volatility process when the dynamics of a logarithmic financial asset price are modeled. Popular examples include, for instance CARMA processes \cite{B01,MS07}, the increments of fractionally integrated L\'evy processes \cite{M06} and non-Gaussian Ornstein-Uhlenbeck type processes \cite{BNS01} where $f(s)=\mathrm{e}^{as}  1_{[0,\infty)}(s)$ with $a \in \R^-$. By allowing $f$ to depend on a random parameter $A$ and replacing the L\'evy process by a L\'evy basis one arrives at so-called mixed moving averages (MMA in short) as for instance in \cite{BN01,BNS11,FS13}.

An important example of MMA are the supOU processes studied in \cite{BN01,BNS11,FK07,FS13}. In the univariate case, assume $\int_{|x|>1} \log(|x|) \, \nu(dx) < \infty$ and $\int_{\R^{-}} -\frac{1}{A} \pi(dA) < \infty$, where $\nu$ is a L\'evy measure and $\pi$ is the probability distribution on $\R^-$ of the random parameter $A$, see Definition 2.1 for details. If $\Lambda$ is a L\'evy basis on $\R$ with those characteristics, then the process
 \[
X_t=\int_{\R^{-}} \int_{-\infty}^t \mathrm{e}^{A (t-s)} \,\Lambda(dA,ds) \,\,\, \forall t \in \R,
\]
is called a supOU process. Whereas a non-Gaussian Ornstein-Uhlenbeck process necessarily exhibits autocorrelation $\mathrm{e}^{ah}$ for $h \in \N$, the supOU process has a flexible dependence structure. For example, its autocorrelations can show a polynomial decay depending on the probability distribution $\pi$. Moreover, when a discrete probability distribution $\pi$ for the random parameter $A$ is considered, we obtain a popular model used, for example, in stochastic volatility models \cite{BN01}, in modeling fractal activity times \cite{K14,L12} and in astrophysics \cite{KTM13}.
  
MMA processes can also be used, under suitable conditions, as building blocks for more complex models.
We study in this paper the class of MMA stochastic volatility models. 
An example of the class is the supOU SV model, defined in \cite{BNS11,BNS13}, where the log-price process (of some financial asset) is defined for $t \in \R^+$ as
\[
J_t=\int_0^t \sqrt{X_s} dW_s, \,\,\, J_0=0,
\] 
and $(W_s)_{s\in\R^+}$ is a standard Brownian motion independent of the process $(X_s)_{s\in \R^+}$ which is a non-negative supOU process. Some examples of applications of the supOU SV model can be found in \cite{BNV13,GS10,SZ15}. 

The aim of this paper is twofold. First, to show that sample moments of an MMA and of the returns of an MMA SV model  have a limiting normal distribution. 
Secondly, to develop a statistical estimation procedure for the MMA and MMA SV model in a semi-parametric framework, where the distribution of the random parameter $A$ is specified in detail, and establishing its asymptotic properties.

For this end, it is of high importance to understand the dependence structure of the class of MMA processes.
In \cite{FS13}, it is shown that an MMA process driven by a L\'evy basis is mixing. However, in order to prove distributional limit theorems which enable valid asymptotic inference stronger notions of asymptotic independence are needed. Often one applies strong mixing properties (see \cite{Doukhan1994,R56}) to this end. Usually they are established by using a Markovian representation and showing geometric ergodicity of it. In turn this requires often smoothness conditions on the driving random noise and it is well-known that even autoregressive processes of order one are not strongly mixing when the distribution of the noise is not sufficiently regular (see \cite{Andrews1984}). We want to obtain results for MMA processes in general, which typically have no suitable Markovian representation, and without regularity conditions on the driving L\'evy basis apart from moment conditions. As will become obvious later on, the weak dependence concepts introduced by Doukhan and Louhichi \cite{DL99} and Doukhan and Dedecker \cite{DD03}, respectively called $\eta$-weak dependence and $\theta$-weak dependence, are very suitable for our purposes. For an extensive introduction on the weak dependence of causal and non-causal processes we refer the reader to \cite{DD08}. We then show the asymptotic normality of the sample mean and the sample autocovariance functions of an MMA process in its non-causal and causal specification. Moreover for the MMA stochastic volatility models, we show the $\theta$-weak dependence of the return process and the distributional limit of its sample moments. 
In \cite{GL17,GL16,GL172}, the limiting behavior of integrated and partial sums of supOU processes is analyzed in relation to the growth rate of their moments, called intermittency when the grow rate is fast. This leads to some conclusions regarding their asymptotic finite dimensional distributions and to identify different limiting theorems depending on the short or long memory shown by the supOU process. In our paper, for short memory supOU processes and more general MMA and MMA SV model we can additionally give, exploiting the weak dependence properties, conditions under which functional central limit theorems hold in distribution as well as consider general moments. 

Later in the paper, we discuss a Generalized Method of Moments (GMM in short) procedure to estimate the parameters of a supOU and supOU SV model. Unfortunately, the classical and efficient maximum-likelihood approach seems not applicable in this case, since the density of the supOU processes is not known in general. However, the supOU process has a known moment structure and GMM estimators can be defined as in \cite{STW15}. In a semiparametric framework, we consider in detail the case in which the random parameter $A$ is Gamma distributed and the moment functions are known in closed form. For the GMM estimators of the supOU process and the return process of a supOU SV model we show the asymptotic normality of both estimators (whose consistency has been shown in \cite{STW15}). Finally, via an explicit computation of the third and fourth order cumulants of the supOU and return process, we give the explicit form of the asymptotic covariance matrices of the GMM estimators. 

Interestingly, our result can also be seen as a first step in obtaining an estimation theory for the ambit processes (homogeneous and stationary) which include an additional multiplicative random input in the definition of an MMA process, see \cite{BBV15,BHSS16,BNSA07}.

The paper is organized as follows. In Section \ref{sec1}, the definition of a L\'evy basis and MMA process is given. In Section \ref{sec2}, the weak dependence properties of an MMA process are discussed. In Section \ref{sec3}, the asymptotic distributions of the moments of non-causal and causal MMA processes are shown.
In Section \ref{sec4}, the definition of an MMA SV model is given and the $\theta$-weak dependence of the return process is analyzed along with its sample moments asymptotic. In Section \ref{sec5}, the asymptotic normality of the GMM estimators of the supOU process and of the supOU SV model is then proven.

\section{L\'evy bases and mixed moving average processes}
\label{sec1}
We start with some preliminary results leading to the definition of an MMA process.
Throughout, we assume that all random variables and processes are defined on a given complete probability space $(\Omega, \mathcal{A},\mathbb{P})$ equipped with a filtration when relevant.
Let $S$ denote a non-empty topological space, $\mathcal{B}(S)$ the Borel $\sigma$-field on $S$, $\pi$ some probability measure on $(S,\mathcal{B}(S))$ and
$\mathcal{B}_b(S \times \R)$ the bounded Borel sets of $S \times \R$. A L\'evy basis, which is also known as an infinitely divisible independently scattered random
measure, is defined as follows.
\begin{Definition}
\label{basis}
A family $\Lambda=\{\Lambda(B):\,B\, \in \mathcal{B}_b(S \times \R) \}$ of $\R^d$-valued random variables is called an $\R^d$-valued L\'evy basis on $S \times \R$ if:
\begin{itemize}
\item the distribution of $\Lambda(B)$ is infinitely divisible for all $B \in \mathcal{B}_b(S\times\R)$,
\item for arbitrary $n \in \N$ and pairwise disjoint sets $B_1,\ldots,B_n \in \mathcal{B}_b(S\times\R)$ the random variables $\Lambda(B_1),\ldots,\Lambda(B_n)$ are independent and 
\item for any pairwise disjoint sets $B_1,B_2,\ldots \in \mathcal{B}_b(S\times\R)$ with $\bigcup_{n \in \N} B_n \in \mathcal{B}_b(S\times\R)$
we have, almost surely, $\Lambda(\bigcup_{n \in \N} B_n)= \sum_{n\in\N} \Lambda(B_n)$.
\end{itemize}
\end{Definition}

We restrict ourselves to time-homogeneous and factorisable L\'evy bases, i.e. L\'evy bases with characteristic function
\begin{equation}
\label{carLB}
\E[\mathrm{e}^{\mathrm{i} \langle u, \Lambda(B) \rangle}]=\mathrm{e}^{\Phi(u)\Pi(B)}
\end{equation}
for all $u \in \R^d$ and $B \in \mathcal{B}_b(S \times \R)$, where $\Pi=\pi\times\lambda$ is the product of a probability measure $\pi$ on $S$ and the Lebesgue measure $\lambda$ on $\R$ and
$$
\Phi(u)=\mathrm{i} \langle \gamma, u \rangle - \frac{1}{2}\langle u, \Sigma u \rangle + \int_{\R^d} \mathrm{e}^{\mathrm{i}\langle u,x \rangle}-1-\mathrm{i}\langle u,x \rangle \mathbb{1}_{[0,1]}(\|x\|) \,\,\nu(dx)
$$

\noindent
is the cumulant transform of an infinitely divisible (i.d. in short) distribution with characteristic triplet $(\gamma, \Sigma,\nu)$, where $\gamma \in \R^d$, $\Sigma \in \mathbb{S}^{+}_d$ - i.e. the space of the positive semi-definite matrices - and $\nu$ is a L\'evy measure - a Borel measure on $\R^d$ with $\nu({0})=0$ and $\int_{\R^d} (\|x\|^2 \wedge 1) \nu(dx) < \infty$ . By $L$ we denote the underlying L\'evy process associated with $(\gamma,\Sigma,\nu)$ and given by $$L_t=\Lambda(S \times (0,t]) \,\, \textrm{and} \,\, L_{-t}=-\Lambda(S \times (-t,0)) \,\, \textrm{for} \,\, t \in \R^{+}.$$ The quadruple $(\gamma, \Sigma,\nu,\pi)$ determines the distribution of the L\'evy basis completely and therefore it is called the generating quadruple.

In the following, norms of vectors or matrices are denoted by $\| \cdot \|$. We are going to work especially with the Euclidean norm or its induced 
operator norm unless otherwise stated. However, due to the equivalence of all norms none of the results in the paper depends on the choice of the norm.
For more information on $\R^d$-valued L\'evy bases see \cite{P03} and \cite{RR89}.

Following \cite{P03}, it can be shown that a L\'evy basis has a L\'evy-It\^o decomposition.

\begin{Theorem}
Let $\Lambda$ be a homogeneous and factorisable $\R^d$-valued L\'evy basis on $S \times\R$ with generating quadruple $(\gamma, \Sigma,\nu,\pi)$. Then there exists a modification $\tilde{\Lambda}$ of $\Lambda$ which is also a L\'evy basis with generating quadruple $(\gamma, \Sigma,\nu,\pi)$ such that there exists an $\R^d$-valued
L\'evy basis $\tilde{\Lambda}^G$ on $S \times \R^d$ with generating quadruple $(0, \Sigma,0,\pi)$ and an independent Poisson random measure $\mu$ on 
$(\R^d \times S\times \R, \mathcal{B}(\R^d \times S\times \R))$ with intensity measure $\nu \times\pi \times \lambda$ which satisfy

\[
\tilde{\Lambda}(B)= \gamma(\pi \times \lambda)(B) + \tilde{\Lambda}^G(B)
 + \int_{\|x\|\leq 1} \int_B x \,(\mu(dx,dA,ds)-ds\pi(dA)\,\nu(dx))
\]
\begin{equation}
\label{itodec}
+ \int_{\|x\|> 1} \int_B x \,\,\mu(dx,dA,ds)
\end{equation}

for all $B \in \mathcal{B}_b( S\times \R)$ and all $\omega \in \Omega$.

Provided $\int_{\|x\|\leq 1} \|x\| \nu(dx) < \infty$, it holds that
\[
\tilde{\Lambda}(B)= \gamma_0(\pi \times \lambda)(B) +\tilde{\Lambda}^G(B) + \int_{\R^d}\int_B x \,\, \mu(dx,dA,ds)
\]
for all $B \in \mathcal{B}_b( S\times \R)$ with
\begin{equation}
\label{yei}
\gamma_0 : = \gamma - \int_{\|x\|\leq 1} x \,\nu(dx).
\end{equation}

Furthermore, the integral with respect to $\mu$ exists as a Lebesgue integral for all $\omega \in \Omega$.
\end{Theorem}
Here an $\R^d$-valued L\'evy basis $\tilde{\Lambda}$ on $S \times \R$ is called a modification of a L\'evy basis 
$\Lambda$ if $\tilde{\Lambda}(B)=\Lambda(B)$ a.s. for all $B \in \mathcal{B}_b(S \times \R)$. We refer the reader to \cite[Section 2.1]{JS03} for
further details on the integration with respect to Poisson random measures.

We also recall the following multivariate extension of \cite[Theorem 2.7]{RR89}.
We denote by $A^{\prime}$ the transpose of a matrix $A$ in what follows.
\begin{Theorem}
\label{uno}
Let $\Lambda$ be an $\R^d$-valued L\'evy basis with generating quadruple $(\gamma,\Sigma, \nu,\pi)$, $f: S \times \R \to M_{n\times d}(\R)$ be a $\mathcal{B}(S \times \R)$-measurable function. Then $f$ is $\Lambda$-integrable as a limit in probability in the sense of Rajput and Rosi\'nski \cite{RR89}, if and only if
{\small \begin{equation} 
\label{ass1}
\int_S \int_{\R} \Big\|f(A,s)\gamma+ \int_{\R^d} f(A,s) x \Big(\mathbb{1}_{[0,1]}(\|f(A,s)x\|)-\mathbb{1}_{[0,1]}(\|x\|)\Big) \,\, \nu(dx) \Big\| \, ds  \, \pi(dA) <\infty,
\end{equation}
\begin{equation}
\label{ass2}
\int_S \int_{\R} \|f(A,s)\Sigma f(A,s)^{\prime}\| \, ds \, \pi(dA) <\infty,
\end{equation}
 and
\begin{equation}
\label{ass3}
\int_S \int_{\R} \int_{\R^d} \Big(1 \wedge \|f(A,s)x\|^2 \Big) \, \nu(dx) \, ds \, \pi(dA) < \infty.
\end{equation}}
If $f$ is $\Lambda$-integrable, the distribution of $\int_S \int_{\R} f(A,s)  \, \Lambda(dA,ds)$ is infinitely divisible with characteristic triplet $(\gamma_{int},\Sigma_{int},\nu_{int})$ given by
{\small\begin{equation}
 \label{gamma}
 \gamma_{int}=\int_S \int_{\R} \Big(f(A,s)\gamma+ \int_{\R^d} f(A,s) x \Big(\mathbb{1}_{[0,1]}(\|f(A,s)x\|)-\mathbb{1}_{[0,1]}(\|x\|)\Big) \,\, \nu(dx) \Big) \, ds \, \pi(dA)
\end{equation}}
\begin{equation}
 \label{sigma}
 \Sigma_{int}= \int_S \int_{\R} f(A,s)\Sigma f(A,s)^{\prime} \, ds \, \pi(dA)
\end{equation}
and
\begin{equation}
 \label{nu}
 \nu_{int}(B) = \int_S \int_{\R} \int_{\R^d} \mathbb{1}_B(f(A,s)x) \, \nu(dx) \, ds \, \pi(dA)
\end{equation}
for all Borel sets $B \subseteq \R^n \setminus \{0\}$.
\end{Theorem}

Implicitly, we assume that $\Sigma_{int}$ or $\nu_{int}$ are different from zero throughout the paper to rule out the deterministic case.

When the underlying L\'evy process has finite variation we can do $\omega$-wise Lebesgue integration; that is, the integral can be obtained as a Lebesgue integral for each $\omega \in \Omega$.

\begin{Corollary}
\label{due}
Let $\Lambda$ be an $\R^d$-valued L\'evy basis with characteristic quadruple $(\gamma,0,\nu,\pi)$ satisfying $\int_{\|x\| \leq 1} \|x\| \nu(dx) < \infty$, and define $\gamma_0$ as in (\ref{yei}), that is $\Phi(u)=i \langle u, \gamma_0 \rangle+ \int (\mathrm{e}^{\mathrm{i}\langle u, x \rangle}-1) \nu(dx)$.
Furthermore, let $f:S \times \R \to M_{n \times d}$ be a $\mathcal{B}(S \times \R)$-measurable function satisfying
\begin{equation}
\label{ass4}
\int_S \int_{\R} \|f(A,s) \gamma_0\| ds  \, \pi(dA) <\infty,
\end{equation}
and
\begin{equation}
\label{ass5}
\int_S \int_{\R} \int_{\R^d} \Big(1 \wedge \|f(A,s)x\| \Big) \, \nu(dx) \, ds \, \pi(dA) < \infty.
\end{equation}
Then,
\[
\int_S \int_{\R} f(A,s) \Lambda(dA,ds) =\int_S \int_{\R} f(A,s) \, \gamma_0 \,ds \pi(dA)
\]
\begin{equation}
\label{yei2}
+\int_{\R^d} \int_S \int_{\R} f(A,s) x \,\mu(dx, dA,ds),
\end{equation}
and the right hand side is a Lebesgue integral for every $\omega \in \Omega$.
Moreover, the distribution $\int_S\int_{\R} f(A,s) \,\Lambda(dA,ds)$ is infinitely divisible with characteristic function
{\small
\[
\mathbb{E}\Big(\exp\Big(\mathrm{i} \langle u, \int_S\int_{\R} f(A,s) \, \Lambda(dA,ds) \rangle\Big) \Big)
= \mathrm{e}^{\mathrm{i} \langle u, \gamma_{int,0} \rangle +\int_{\R^n} (\mathrm{e}^{\mathrm{i} \langle u,x \rangle}-1) \nu_{int}(dx)} \,\,\, u\in\R^n,
\]}
where
\[
\gamma_{int}=\int_S \int_{\R} f(A,s) \gamma_0 \,\, ds \, \pi(dA),
\]
\[
\nu_{int}(B)=\int_S \int_{\R} \int_{\R^d} \mathbb{1}_B(f(A,s)x) \,\,\nu(dx)\, ds \, \pi(dA) \,\,\, \forall B \in \R^n \setminus \{0\}.
\]
\end{Corollary}

The above corollary follows immediately from the L\'evy-It\^o decomposition (\ref{itodec}) and the usual integration theory with respect to a Poisson random measure.
We notice that the result (\ref{yei2}) is an immediate consequence of working with an underlying L\'evy process of finite variation, as no compensation for the small jumps is needed if $\int_{\|x\| \leq 1} \|x\| \nu(dx) < \infty$.

We can now introduce an MMA process driven by a L\'evy basis.
\begin{Definition}
 Let $\Lambda$ be an $\R^d$-valued L\'evy basis on $S \times \R$ and let $f: S \times \R \to M_{n\times d}(\R)$ be a $\mathcal{B}(S \times \R)$-measurable function
satisfying assumptions (\ref{ass1}), (\ref{ass2}) and (\ref{ass3}). Then, the process 
 \begin{equation}
 \label{mma}
 X_t\colon = \int_S \int_{\R} f(A,t-s)\,\, \Lambda(dA,ds) 
 \end{equation}
is well defined for each $t\in \R$, infinitely divisible and strictly stationary. It is called a $n$-dimensional mixed moving average process and $f$ its kernel function. 
\end{Definition}

We conclude the section giving sufficient conditions ensuring the finiteness of moments of an MMA process.

\begin{Proposition}
\label{moment1}
Let $X$ be an n-dimensional MMA process driven by a L\'evy basis $\Lambda$ satisfying the conditions of Theorem \ref{uno}.
\begin{enumerate}[(i)]
\item If $\int_{\|x\|>1} \| x\|^r \,\nu(dx) < \infty \, \textrm{and} \, f \in L^r(S \times \R, \pi \otimes \lambda)$
for $r \in [2, \infty)$, then $\E[\|X_t\|^r]<\infty$.
\item If $\int_{\|x\|>1} \| x\|^r \,\nu(dx) < \infty \, \textrm{and} \, f \in L^r(S \times \R, \pi \otimes \lambda)\cap L^2(S \times \R, \pi \otimes \lambda) $ for $r \in (0, 2)$, then $\E[\|X_t\|^r]<\infty$.
\end{enumerate}
\end{Proposition}

\begin{proof}
Following \cite[Corollary 25.8]{SA}, we have to show that $\int_{\|x\|>1} \| x\|^r \,\nu_{int}(dx) < \infty$. 
Since

\[
\int_{\|x\|>1} \| x\|^r \,\nu_{int}(dx)
\]
\[
= \int_S \int_{\R} \int_{\R^d}  \|f(A,s) x\|^r 1_{(1,\infty)} (\|f(A,s)x\|) \nu(dx) ds \pi(dA)
\]
\[
\leq \int_S \int_{\R} \int_{\R^d}  \|f(A,s)\|^r \|x\|^r 1_{(1,\infty)} (\|x\|) \nu(dx) ds \pi(dA)
\]
\[
+\int_S \int_{\R} \int_{\R^d}  \|f(A,s)\|^{r \vee 2} \|x\|^{r \vee 2} 1_{(0,1)}(\|x\|) 1_{(1,\infty)}(\|f(A,s)x\|)  \nu(dx) ds \pi(dA),
\]
we can conclude that (i) and (ii) follow, given that $\nu$ is a L\'evy measure.

\end{proof}

If the underlying L\'evy process $L$ is of finite variation, an analogous proof gives the following results.

\begin{Corollary}
\label{moment2}
Let $X$ be an n-dimensional MMA process driven by a L\'evy basis $\Lambda$ satisfying the conditions of Corollary \ref{due}.
\begin{enumerate}[(i)]
\item If $\int_{\|x\|>1} \| x\|^r \,\nu(dx) < \infty \, \textrm{and} \, f \in L^r(S \times \R, \pi \otimes \lambda)$
for $r \in [1, \infty)$, then $\E[\|X_t\|^r]<\infty$.
\item If $\int_{\|x\|>1} \| x\|^r \,\nu(dx) < \infty \, \textrm{and} \, f \in L^r(S \times \R, \pi \otimes \lambda)\cap L^1(S \times \R, \pi \otimes \lambda) $ for $r \in (0, 1)$, then $\E[\|X_t\|^r]<\infty$.
\end{enumerate}
\end{Corollary}

\section{Weak dependence properties of a mixed moving average process}
\label{sec2}

Let  $(\mathcal{A}_t)_{t \in \R}$ be the filtration generated by $\Lambda$ defined as the $\sigma$-algebras $\mathcal{A}_t$ generated by the set of random variables $\{\Lambda(B): B \in \mathcal{B}(S \times (-\infty,t])\}$ for $t \in \R$. If an MMA process is adapted to $(\mathcal{A}_t)_{t \in \R}$, we call it causal. Otherwise it is referred to as being non-causal.

In the following we will refer by $\N$ to the set of the non negative integers, by $\N^*$ to the set of the positive integers, by $\R^-$ to the set of negative real numbers and by $\R^+$ to the set of the non negative real numbers.

\subsection{Non-causal case}

Let 
\[
\mathcal{F}=\bigcup_{u \in \mathbb{N}^*} \mathcal{F}_u 
\]
where $\mathcal{F}_u$ is the class of bounded functions from $(\R^n)^u$ to $\R$ Lipschitz with respect to a distance $\delta_1$ on $(\R^n)^u$ defined by
\begin{equation}
\label{dist}
\delta_1(x^*,y^*)= \sum_{i=1}^u \delta(x_i,y_i),
\end{equation}
where $x^*=(x_1,\ldots,x_u)$ and $y^*=(y_1,\ldots,y_u)$ and $x_i,y_i \in \R^n$ for all $i=1,\ldots,u$.
We consider $\R^n$ equipped with the Euclidean norm and then $\delta(x_i,y_i)= \|x_{i}-y_{i}\|$.

\begin{Definition}
\label{eta_g}
 A process $X=(X_t)_{t\in \R}$ with values in $\R^n$ is called an $\eta$-weakly dependent process if there exists a sequence $(\eta(r))_{r \in \R^{+}}$ converging to $0$,
 satisfying
 {\small
 \begin{equation}
 \label{def}
 |Cov(F(X_{i_1},\ldots,X_{i_u}),G(X_{j_1},\ldots,X_{j_v}))|\leq c \, (u Lip(F) \|G\|_{\infty} +v Lip(G) \|F\|_{\infty}) \eta(r)
 \end{equation}}
for all
\begin{equation*}
\left\{
  \begin{array}{l}
(u,v) \in \N^* \times \N^*;\\
r\in \R^{+}; \\
(i_1,\ldots,i_u) \in \R^u \,\, \textrm{and}\,\, (j_1,\ldots,j_v) \in \R^v, \\ \textrm{with}\,\, i_1\leq\ldots\leq i_u \leq i_u+r\leq j_1\leq \ldots\leq j_v;\\
\textrm{functions} \,\, F \colon (\R^{n})^u \to \R \,\, \textrm{and}\,\, G\colon (\R^{n})^v \to \R \,\,\textrm{belonging to $\mathcal{F}$ and satisfying}\\
\|G\|_{\infty}\leq 1, \, \|F\|_{\infty} \leq 1\,\, \mathrm{and} \,\,  Lip(F) +Lip(G) < \infty, \\
\mathrm{where} \,\,  Lip(F)=\sup_{x\neq y} \frac{|F(x)-F(y)|}{\| x_1-y_1 \|+\|x_2-y_2\|+ \ldots+ \|x_n-y_n\|},
\end{array}
\right.
\end{equation*}
and where $c$ is a constant independent of $r$.
We call $(\eta(r))_{r \in \R^{+}}$ the sequence of the $\eta$-coefficients.  
\end{Definition}

The above definition makes the asymptotic independence between {\it past} and {\it future} explicit, this means that the {\it past} is progressively forgotten. In terms of the initial process $X$, {\it past}  and {\it future} are elementary events respectively defined through finite-dimensional marginals
as $A_u= (X_{i_1},\ldots,X_{i_u})$ and $B_v=(X_{j_1},\ldots,X_{j_v})$ for $i_1\leq\ldots\leq i_u \leq i_u+r\leq j_1\leq \ldots\leq j_v$ and $r \geq 0 $.

The weak dependence property, as stated in Definition \ref{eta_g}, depends upon the class of functions $\mathcal{H}=\{f \in \mathcal{F}: \|f\|_{\infty}\leq 1\}$. However, it can also be defined in $\mathcal{F}$ as discussed in \cite{DL99}. 
Note, a similar definition can be given for the strong mixing property introduced by Rosenblatt \cite{R56}.
Let $\sigma(A_u)$ and $\sigma(B_v)$ be the $\sigma$-algebras generated by the finite-dimensional marginals $A_u$ and $B_v$ and $\mathcal{F}^*=\bigcup_{u \in \mathbb{N}} \mathcal{F}^*_u$ where $\mathcal{F}^*_u$ is the class of bounded functions from $(\R^n)^u$ to $\R$. We define
\begin{equation}
\label{mixing}
\alpha(\sigma(A_u),\sigma(B_v))= \sup_{f,g \in \mathcal{H}^*} |Cov(f(A_u),g(B_v))|,
\end{equation}
where $\mathcal{H}^*=\{f \in \mathcal{F}^*:\|f\|_{\infty}\leq 1 \}$,
and then $\alpha$-strong mixing coefficient is
\[
\alpha(r)=\sup_{u,v} \max_{\begin{array}{l}
i_1\leq\ldots\leq i_u\\
j_i\leq\ldots\leq j_v\\
r=j_1-i_u 
\end{array}} \alpha(\sigma(A_u),\sigma(B_v)).
\]

We notice that definition (\ref{def}) holds for a set of functions in $\mathcal{H}$ whereas (\ref{mixing}) holds for functions belonging to $\mathcal{H}^*$.
This means that if a process $X$ is strongly mixing then it is also weakly dependent but the reverse implication does not necessarily hold. 
The only known case of the equivalence of the two definitions can be found in \cite[Proposition 1]{DF12} where it is shown that an $\eta$-weakly dependent integer valued process satisfies the strong mixing condition.

We now show that a non-causal MMA process is an $\eta$-weakly dependent process.

\begin{Proposition}
\label{tre}
Let $\Lambda$ be an $\R^d$-valued L\'evy basis with characteristic quadruple $(\gamma,\Sigma,\nu,\pi)$ such that $\E[L_1]=0$ and $\int_{\|x\| >1} \|x\|^2 \nu(dx) < \infty$, $f: S \times \R \to M_{n \times d}(\R)$ a $\mathcal{B}(S \times \R)$-measurable function and $f \in L^2(S \times \R, \pi \otimes \lambda)$. Then, the resulting MMA process $X$ is an $\eta$-weakly dependent process with coefficients
\begin{equation}
\label{coefficients}
\eta_X(r)= \Big( \int_S \int_{\R} tr(f(A,-s)\Sigma_L f(A,-s)^{\prime}) \mathbb{1}_{(-\infty,-r)}(2s) \, ds \, \pi(dA) \Big)^{\frac{1}{2}}
\end{equation}
\[
+ \Big( \int_S \int_{\R} tr(f(A,-s)\Sigma_L f(A,-s)^{\prime}) \mathbb{1}_{(r,+\infty)}(2s) \, ds \, \pi(dA) \Big)^{\frac{1}{2}},
\]
for all $r \geq 0$, where $\E[L_1L_1^{\prime}]=\Sigma_L=\Sigma+\int_{\R^d} x x^{\prime} \nu(dx)$.
\end{Proposition}

\begin{proof}
First, we define $\forall t \in \R$ and $m \geq 0$ the truncated sequence\begin{equation}
\label{truncated}
X_t^{(m)}=\int_S  \int_{\R} f(A,t-s) \mathbb{1}_{[-m,m]}(t-s) \, \Lambda(dA,ds)= \int_S \int_{t-m}^{t+m} f(A,t-s) \,\, \Lambda(dA,ds).
\end{equation}
Since the kernel function $f$ is square integrable, we have that properties (\ref{ass2}) and (\ref{ass3}) hold and so $f$ is $\Lambda$-integrable (Theorem \ref{uno}) and $X$ is well defined. Moreover, Proposition \ref{moment1} holds, $\E[X_t^2] < \infty$ for all $t \in \R$ and we can determine an upper bound of the expectation 
\[
\E\|X_t-X_t^{(m)}\|= \E \Big \| \int_S \int_{-\infty}^{t-m} f(A,t-s) \, \Lambda(dA,ds) + \int_S \int_{t+m}^{+\infty} f(A,t-s) \Lambda(dA,ds) \Big \| 
\]
\[
\leq \Big(\E \Big \| \int_S \int_{-\infty}^{t-m} f(A,t-s) \, \Lambda(dA,ds) \Big \|^2 \Big)^{\frac{1}{2}} 
+ \Big(\E \Big \| \int_S \int_{t+m}^{\infty} f(A,t-s) \, \Lambda(dA,ds)\Big \|^2 \Big)^{\frac{1}{2}}.
\]
Due to the stationarity of $X$ the above estimation is independent of $t$ and equal to
\begin{align}
\label{res}
\begin{split}
\Big( \int_S\int_{-\infty}^{-m} tr(f(A,-s)\Sigma_L f(A,-s)^{\prime}) \, ds \pi(dA) \Big)^{\frac{1}{2}}\\
+\Big( \int_S\int_{m}^{\infty} tr(f(A,-s)\Sigma_L f(A,-s)^{\prime}) \, ds \pi(dA) \Big)^{\frac{1}{2}}
\end{split}
\end{align}

Let $F$ and $G$ belong to the class of bounded functions $\mathcal{H}$, $(u,v) \in \mathbb{N}^* \times \mathbb{N}^*$, $r \in \R^{+}$, $(i_1,\ldots,i_u) \in \R^u$ and $(j_1,\ldots,j_v) \in \R^v$ with $i_1\leq\ldots\leq i_u \leq  i_u+r\leq j_1\leq \ldots\leq j_v$, $X_i^*=(X_{i_1},\ldots,X_{i_u})$,
$X_j^*=(X_{j_1},\ldots,X_{j_v})$, and $X_i^{*(m)}=(X_{i_1}^{(m)},\ldots,X_{i_u}^{(m)})$, $X_j^{*(m)}=(X_{j_1}^{(m)},\ldots,X_{j_v}^{(m)})$ where for all $m \geq 0$
\begin{equation*}
X_{i_u}^{(m)}=\int_S\int_{i_u-m}^{i_u+m} f(A,i_u-s) \, \Lambda(dA,ds) 
\end{equation*}
\begin{equation}
\label{ind1}
\textrm{and} \hspace{9.6cm}
\end{equation}
\begin{equation*}  
X_{j_1}^{(m)}=\int_S\int_{j_1-m}^{j_1+m} f(A,j_1-s) \, \Lambda(dA,ds).
\end{equation*}
Then, if $j_1-m-i_u-m \geq 0$, that can also be expressed as $j_1-i_u \geq 2m$, $I_u=S \times [i_u-m,i_u+m]$ and $J_1=S \times [j_1-m,j_1+m]$ are disjoint sets or they have intersection $S\times\{j_1-m\}$ when $j_1-i_u=2m$. Noting that $\pi \times \lambda (S\times \{j_1-m\})=0$ by the definition of a L\'evy basis, the two sequences $(X_i^{(m)})_{i_1,\ldots,i_u}$ and $(X_j^{(m)})_{j_1,\ldots,j_v}$ are independent and so are $F(X_i^{*(m)})$ and $G(X_j^{*(m)})$.
Therefore,
\[
|Cov(F(X_i^*),G(X_j^*))|
\]
\[
\leq |Cov(F(X_i^*)-F(X_i^{*(m)}),G(X_j^*))|+|Cov(F(X_i^{*(m)}), G(X_j^{*})-G(X_j^{*(m)}))|
\]
\[
\leq 2(\E |F(X_i^*)-F(X_i^{*(m)})|+ \E |G(X_j^*)-G(X_j^{*(m)})| )
\]
the last relation comes from $\|F\|_{\infty}, \|G\|_{\infty}\leq 1$
\[
\leq 2 \Big(Lip(F) \sum_{l=1}^u \E\|X_{i_l}-X_{i_l}^{(m)}\|+ Lip(G) \sum_{k=1}^v \E\|X_{j_k}-X_{j_k}^{(m)}\| \Big)
\]
using the result (\ref{res}) for $m=\frac{r}{2}$
\[
\leq 2 (u Lip(F)+ v Lip(G)) \Big \{ \Big( \int_S\int_{\R} tr(f(A,-s)\Sigma_L f(A,-s)^{\prime}) \, \mathbb{1}_{(-\infty, -r) }(2s) \, ds \pi(dA) \Big)^{\frac{1}{2}}
\]
\[
+ \Big( \int_S\int_{\R} tr(f(A,-s)\Sigma_L f(A,-s)^{\prime}) \, \mathbb{1}_{(r, +\infty) }(2s) \, ds \pi(dA) \Big)^{\frac{1}{2}} \Big\}
\]
\[
=2(u Lip(F) + v Lip(G)) \eta_X(r),
\]
which converges to zero as $r$ goes to infinity by applying the dominated convergence theorem.
\end{proof}

The following Corollary establishes the $\eta$ weak dependence of an MMA when its underlying L\'evy process has finite mean possibly different from zero.
\begin{Corollary}
\label{tre_non}
Let $\Lambda$ be an $\R^d$-valued L\'evy basis with characteristic quadruple $(\gamma,\Sigma,\nu,\pi)$ and $\int_{\|x\| >1} \|x\|^2 \nu(dx) < \infty$, $f: S \times \R \to M_{n \times d}(\R)$ a $\mathcal{B}(S \times \R)$-measurable function satisfying assumption (\ref{ass1}) and $f \in L^2(S \times \R, \pi \otimes \lambda)$. Then, the resulting MMA process $X$ is an $\eta$-weakly dependent process with coefficients

\begin{align}
\label{eta_mean}
\begin{split}
\eta_{X}(r)=& \Big( \int_S \int_{\R} tr(f(A,-s)\Sigma_L f(A,-s)^{\prime}) \mathbb{1}_{(-\infty,-r)}(2s) \, ds \, \pi(dA)\\
+& \Big \|\int_S \int_{\R} f(A,-s) \mu \, \mathbb{1}_{(-\infty,-r)}(2s) \, ds \, \pi(dA) \Big \|^2 \Big)^{\frac{1}{2}}\\
+&\Big( \int_S \int_{\R} tr(f(A,-s)\Sigma_L f(A,-s)^{\prime}) \mathbb{1}_{(r,+\infty)}(2s) \, ds \, \pi(dA)\\
+& \Big \|\int_S \int_{\R} f(A,-s) \mu \, \mathbb{1}_{(r,+\infty)}(2s) \, ds \, \pi(dA) \Big \|^2 \Big)^{\frac{1}{2}}
\end{split}
\end{align}
for all $r \geq 0$, where $E[L_1]= \mu= \gamma+ \int_{\|x\|> 1} x \,\nu(dx)$ and $\E[L_1L_1^{\prime}]=\Sigma_L=\Sigma+\int_{\R^d} x x^{\prime} \,\nu(dx)$.
\end{Corollary}

\begin{proof}
Let $X_t^{(m)}$ defined as in Proposition \ref{tre} for all $t \in \R$ and $m \geq 0$. Then, (\ref{res}) becomes
\begin{align*}
\begin{split}
&\Big( \int_S \int_{\R} tr(f(A,-s)\Sigma_L f(A,-s)^{\prime}) \mathbb{1}_{(-\infty,-m)}(s) \, ds \, \pi(dA)\\&+ \Big \|\int_S \int_{\R} f(A,-s) \mu \mathbb{1}_{(-\infty,-m)}(s) \, ds \, \pi(dA) \Big \|^2 \Big)^{\frac{1}{2}}\\
&+\Big( \int_S \int_{\R} tr(f(A,-s)\Sigma_L f(A,-s)^{\prime}) \mathbb{1}_{(m,+\infty)}(s) \, ds \, \pi(dA)\\ 
&+\Big \|\int_S \int_{\R} f(A,-s) \mu \mathbb{1}_{(m,+\infty)}(s) \, ds \, \pi(dA) \Big \|^2 \Big)^{\frac{1}{2}}
\end{split}
\end{align*}
Proceeding as in the proof of Proposition \ref{tre}, the $\eta$-coefficients (\ref{eta_mean}) are obtained.
\end{proof}

When the underlying L\'evy process is of finite variation we can lighten the moment assumptions on the MMA process.
The below result applies to all finite variation MMA process with finite mean.

\begin{Corollary}
\label{quattro}
Let $\Lambda$ be an $\R^d$-valued L\'evy basis with characteristic quadruple $(\gamma,0,\nu,\pi)$ such that $\int_{\R^d} \|x\| \nu(dx) < \infty$, $f: S \times \R \to M_{n \times d}(\R)$ a $\mathcal{B}(S \times \R)$-measurable function and $f \in L^1(S \times \R, \pi \otimes \lambda)$ and define $\gamma_0$ as in (\ref{yei}).
Then, the resulting MMA process $X$ is an $\eta$-weakly dependent process with coefficients
\begin{equation}
\label{coefficients2}
\eta_X(r)=  \int_S \int_{\R} \int_{\R^d} \| f(A,-s)x\| \, \mathbb{1}_{(-\infty,-r)\bigcup (r,+\infty)}(2s) \,\nu(dx)\, ds \, \pi(dA)
\end{equation}
\[
+\int_S \int_{\R} \| f(A,-s) \gamma_0 \| \, \mathbb{1}_{(-\infty,-r)\bigcup (r,+\infty)}(2s) \, ds \, \pi(dA),
\]
for all $r \geq 0$.
\end{Corollary}

\begin{proof}
As the kernel function $f$ is in $L^1$, the properties (\ref{ass4}) and (\ref{ass5}) are satisfied. Thus, $f$ is $\Lambda$-integrable and $X$ well defined. 
Moreover, because of Corollary \ref{moment2}, $\E[X_t]< \infty$.
Using the notation of Proposition \ref{tre}, for all $t \in \R$ and $m \geq 0$
\[
\E \|X_t-X_t^{(m)}\|  \leq \int_S \int_{\R} \int_{\R^d} \|f(A,-s)x\| \mathbb{1}_{(-\infty,-m)\bigcup (m,+\infty)}(s) \,\nu(dx)\, ds \, \pi(dA),
\]
\[
+ \int_S \int_{\R} \| f(A,-s) \gamma_0 \| \, \mathbb{1}_{(-\infty,-m)\bigcup (m,+\infty)}(s) \, ds \, \pi(dA),
\]
where $X_t^{(m)}$ is the truncated sequence (\ref{truncated}).
Thus, for $m=\frac{r}{2}$ and $F$, $G$, $X_i^*$ and $X_j^*$ 
\[
|Cov(F(X_i^*),G(X_j^*))| 
\]
\[
\leq 2 (u Lip(F)+ v Lip(G)) \int_S \int_{\R} \int_{\R^d} \| f(A,-s)x\| \, \mathbb{1}_{(-\infty,-r)\bigcup (r,+\infty)}(2s) \,\nu(dx)\, ds \, \pi(dA) 
\]
\[
+ \int_S \int_{\R} \| f(A,-s) \gamma_0 \| \, \mathbb{1}_{(-\infty,-r)\bigcup (r,+\infty)}(2s) \, ds \, \pi(dA)
\]
\[
=2 (u Lip(F)+ v Lip(G)) \eta_X(r).
\]
Finally, we conclude by applying the dominated convergence theorem.
\end{proof}

The $\eta$ coefficients have some hereditary properties. For example, let $h:\R^n \to \R$ be a Lipschitz function, then
if the sequence $(X_t)_{t \in \R}$ is $\eta$-weakly dependent, the same is true for the sequence $(h(X_t))_{t \in \R}$. 
The latter can be readily checked directly based on Definition \ref{eta_g}.
Hereditary properties for functions that are not Lipschitz on the whole space $\R^n$ can be found in \cite{BDL2007} Lemma $6$, for stationary processes. 
Below follows a multivariate extension of this Lemma for $h:\R^n \to \R^m$. 
\begin{Proposition}
\label{her2}
Let $(X_t)_{t \in \R}$ be an $\R^n$-valued stationary process and assume there exists some constant $C>0$ such that $ \E[\|X_{0}\|^p]^{\frac{1}{p}} \leq C,$ with $p>1$, $h \colon \R^n \to \R^m$ be a function such that $h(0)=0$, $h(x)=(h_1(x),\ldots,h_m(x))$ and
\begin{equation}
\label{hyp_her2}
\|h(x)-h(y)\| \leq c \,\|x-y\| (1+\|x\|^{a-1}+\|y\|^{a-1}),
\end{equation}
for $x,y \in \R^n$, $c>0$ and $1\leq a < p$. 
Define $(Y_t)_{t \in \R}$ by $Y_t=h(X_t)$. If  $(X_t)_{t \in \R}$ is an $\eta$-weakly dependent process,
then $(Y_t)_{t \in \R}$ is an $\eta$-weakly dependent process such that
\[
\forall \, r \geq 0, \,\,\, \eta_Y(r)= \mathcal{C} \, \eta_X(r)^{\frac{p-a}{p-1}},
\]
with the constant $\mathcal{C}$ independent of $r$. 
\end{Proposition}

\begin{proof}
	
	For $(u,v) \in \N^*\times \N^*$, $(i_1,\ldots,i_u) \in \R^u$ and $(j_1,\ldots,j_v) \in \R^v$ with $i_1\leq\ldots\leq i_u \leq i_u+r\leq j_1\leq \ldots\leq j_v$, let us call
	\[
	Y_i^*=(h(X_{i_1}),\ldots,h(X_{i_u})),\,\,Y_{j}^*=(h(X_{j_1}),\ldots,h(X_{j_v})).
	\]
	
	Let $F:\R^{mu}\to \R$, $G:\R^{mv} \to \R$ bounded, such that $\|F\|_{\infty},\|G\|_{\infty} \leq 1$, and Lipschitz functions with respect to the distance (\ref{dist}), then 
	\[
	F(Y_i^*)=F(h(X_{i_1}),\ldots,h(X_{i_u})),\,\,G(Y_j^*)=F(h(X_{j_1}),\ldots,h(X_{j_v})),
	\]
	and
	\[
	F^{(M)}(Y_i^*)=F(h(X_{i_1}^{(M)}),\ldots,h(X_{i_u}^{(M)})),\,\,G^{(M)}(Y_j^*)=F(h(X_{j_1}^{(M)}),\ldots,h(X_{j_v}^{(M)})),
	\]
	where $X_{i}^{(M)}=X_i \mathbb{1}_{\|X_i\|\leq M}$ and w.l.o.g $M>1$.
	According to Definition \ref{eta_g}, we start by analyzing 
	\begin{align}
		\label{bound}
		\begin{split}
			&|Cov(F(Y_i^*),G(Y_j^*))| \leq |Cov(F(Y_i^*)-F^{(M)}(Y_i^*),G(Y_j^*))|\\
			&+|Cov(F^{(M)}(Y_i^*),G(Y_j^*)-G^{(M)}(Y_j^*))|+|Cov(F^{(M)}(Y_i^*), G^{(M)}(Y_j^*))|.
		\end{split}
	\end{align}
	We have that
	\begin{equation}
		\label{est1}
		|Cov(F(Y_i^*)-F^{(M)}(Y_i^*),G(Y_j^*))|\leq 2 \|G\|_{\infty} \E|F(Y_i^*)-F^{(M)}(Y_i^*)|
	\end{equation}
	\[
	\leq 2 Lip(F) \sum_{l=1}^u \E \|h(X_{i_l})-h(X_{i_l}^{(M)})\|.
	\]
	By assumption, for each $l=1,\ldots,u$
	\begin{equation}
		\label{conto1}
		\E (\|h(X_{i_l})-h(X_{i_l}^{(M)})\|) \leq c \,\E (\|X_{i_l}-X_{i_l}^{(M)}\| ( 1+\|X_{i_l}\|^{a-1} + \|X_{i_l}^{(M)}\|^{a-1}))
	\end{equation}
	\[
	\leq c \, \E (\| X_{i_l} \| \mathbb{1}_{\|X_{i_l}\| >M})+ c \,\E (\|X_{i_l}\|^{a-1} \| X_{i_l} \| \mathbb{1}_{\|X_{i_l}\| >M})
	\]
	\[
	+ c \,\E (\| X_{i_l} \| \mathbb{1}_{\|X_{i_l}\| >M}) M^{a-1}
	\]
	\[
	\leq c \, \E \Big(\| X_{i_l} \| \frac{\| X_{i_l} \|^{p-1}}{M^{p-1}} \mathbb{1}_{\|X_{i_l}\| >M} \Big)+ 
	c \,\E \Big(\|X_{i_l}\|^{a-1} \frac{\| X_{i_l} \|^{p-a} }{M^{p-a}} \| X_{i_l} \| \mathbb{1}_{\|X_{i_l}\| >M}\Big)
	\]
	\[
	+ c \,\E \Big(\| X_{i_l} \| \frac{\| X_{i_l} \|^{p-1}}{M^{p-1}} \mathbb{1}_{\|X_{i_l}\| >M}\Big) M^{a-1}
	\]
	\[
	\leq c \, \E( \| X_{i_l} \|^p \mathbb{1}_{\|X_{i_l}\| >M}) M^{1-p} + 2 c \, \E( \| X_{i_l} \|^p \mathbb{1}_{\|X_{i_l}\| >M}) M^{a-p}
	\]
	\[
	\leq 3c\, \E (\|X_{i_l}\|^p)  M^{a-p}.
	\]
	
	Therefore (\ref{est1}) is less than or equal to $6 c \,u\,Lip(F) C^p M^{a-p}$. An analogous bound holds for $|Cov(F^{(M)}(Y_i^*),G(Y_j^*)-G^{(M)}(Y_j^*))|$.
	Moreover, $F^{(M)}$ is a Lipschitz function on the set $\mathcal{A}=\{x=(x_1,\ldots, x_u) \in  \R^{nu}: \|x_i\|\leq M \, \textrm{for}\, i=1,\ldots,u \}$. 
	Let $Z^{(M)}, W^{(M)} \in \mathcal{A}$, then 
	\begin{align*}
		&|F(h(Z^{(M)}_1),\ldots,h(Z^{(M)}_u))-F(h(W^{(M)}_1), \ldots, h(W^{(M)}_u))|\\
		&\leq Lip(F) \sum_{l=1}^u \|h(Z_{l}^{(M)})-h(W_{l}^{(M)})\| \\
		&\leq c \, Lip(F) \sum_{l=1}^u  \|Z_{l}^{(M)} -W_{l}^{(M)}\|(1+ \| Z_{l}^{(M)}\|^{a-1} \|W_l^{(M)}\|^{a-1})\\
		&\leq 3c \, Lip(F) M^{a-1} \sum_{l=1}^u  \|Z_{l}^{(M)} -W_{l}^{(M)}\|.
	\end{align*}
	The same argument holds also for the function $G^{(M)}$. 
	
	$X_{t}^{(M)}$ is a process with values in $\mathcal{A}$ and $\eta$-weakly dependent with the same coefficients as $X_t$, then 
	\[
	|Cov(F^{(M)}(Y_i^*), G^{(M)}(Y_j^*))|=
	\]
	\[
	\leq 3 c \,( u Lip(F)  +  v Lip(G) ) M^{a-1} \eta_X(r).
	\]
	To conclude, (\ref{bound}) is less than or equal to
	\[
	6 c \,(u Lip(F) + v Lip(G)) \Big( \frac{M^{a-1}}{2} \eta_X(r) + C^p M^{a-p} \Big).
	\]
	By choosing  $M= \eta_X(r)^{\frac{1}{1-p}}$ and calling $\mathcal{C}=6c (C^p+\frac{1}{2})$, we obtain that
	\[
	\eta_Y(r)= \mathcal{C} \, \eta_X(r)^{\frac{p-a}{p-1}}.
	\]

\end{proof}

For a polynomial function $h(x)$ we have 
\begin{Corollary}
\label{hereditary_lemma}
Let $(X_t)_{t \in \R}$ be an $\R^n$-valued stationary process and assume there exists some constant $C>0$ such that $ \E[\|X_{0}\|^p]^{\frac{1}{p}} \leq C,$ with $p>1$, $h \colon \R^n \to \R^m$ be a function such that $h(0)=0$ and $h(x)=(h_1(x),\ldots,h_m(x))$ with $h_s(\cdot)$ for $s=1,\ldots,m$ being a polynomial with degree at 
most $a$ for $1\leq a < p$. 
Define $(Y_t)_{t \in \R}$ by $Y_t=h(X_t)$ for $t \in \R$ an $\R^m$-valued process. If  $(X_t)_{t \in \R}$ is an $\eta$-weakly dependent process, then $(Y_t)_{t \in \R}$ is an $\eta$-weakly dependent process such that
\[
\forall \, r \geq 0, \,\,\, \eta_Y(r)= \mathcal{C} \, \eta_X(r)^{\frac{p-a}{p-1}},
\]
with the constant $\mathcal{C}$ independent of $r$. 
\end{Corollary}

\begin{proof}
The function $h$ satisfies the assumption (\ref{hyp_her2}) for each polynomial degree $a$ less than $p$.
Proposition \ref{her2} can then be applied. 
\end{proof}

\subsection{Causal case}

\begin{Definition}
\label{theta_g}
 A process $X=(X_t)_{t\in \R}$ with values in $\R^n$ is called a $\theta$-weakly dependent process if there exists a sequence $(\theta(r))_{r \in \R^{+}}$ converging to $0$,
 satisfying
 \begin{equation}
 \label{def2}
 |Cov(F(X_{i_1},\ldots,X_{i_u}),G(X_{j_1},\ldots,X_{j_v}))|\leq c\, (v Lip(G) \|F\|_{\infty}) \theta(r)
 \end{equation}
for all
\begin{equation*}
\left\{ \begin{array}{l}
(u,v) \in \N^* \times \N^*,\\
r \in \R^{+};\\
(i_1,\ldots,i_u) \in \R^u \,\, \textrm{and}\,\, (j_1,\ldots,j_v) \in \R^v, \\ \textrm{with}\,\, i_1\leq\ldots\leq i_u \leq i_u+r\leq j_1\leq \ldots\leq j_v;\\
\textrm{functions} \,\, F \colon (\R^{n})^u \to \R \,\, \textrm{and}\,\, G\colon (\R^{n})^v \to \R \,\,\textrm{respectively belonging to $\mathcal{H}^*$ and $\mathcal{H}$},\\
\mathrm{where} \,\, Lip(G)=\sup_{x\neq y} \frac{|G(x)-G(y)|}{\| x_1-y_1 \|+\|x_2-y_2\|+ \ldots+ \|x_n-y_n\|},
\end{array}
\right.
\end{equation*}
and where $c$ is a constant independent of $r$.
We call $(\theta(r))_{r \in \R^{+}}$ the sequence of the $\theta$-coefficients. 
\end{Definition}

The $\theta$-weak dependence condition is stronger than the one for $\eta$-weak dependence.
Hence, moment conditions and decay demands on the rate of the $\theta$-coefficients for central limit theorems are typically weaker than in the case of $\eta$-weak dependence, see \cite{DD03}.
It should be also noticed that $\eta(r) \leq \theta(r)$ and that in the case of integer valued processes, \cite{DF12}, the $\theta$-weak dependence implies the strong mixing condition. 

A causal MMA process is defined as follows
\begin{Definition}
 Let $\Lambda$ be an $\R^d$-valued L\'evy basis on $S \times \R^+$ and let $f: S \times \R^+ \to M_{n\times d}(\R)$ be a $\mathcal{B}(S \times \R^+)$-measurable function
satisfying assumptions (\ref{ass1}), (\ref{ass2}) and (\ref{ass3}). Then, the process 
 \begin{equation}
 \label{mma_causal}
 X_t\colon = \int_S \int_{-\infty}^t f(A,t-s)\,\, \Lambda(dA,ds) 
 \end{equation}
is well defined for each $t\in \R$, infinitely divisible and strictly stationary. It is called a causal $n$-dimensional mixed moving average process and $f$ its kernel function. 
\end{Definition}

\begin{Proposition}
\label{tre_theta}
Let $\Lambda$ be an $\R^d$-valued L\'evy basis with characteristic quadruple $(\gamma,\Sigma,\nu,\pi)$ such that $\E[L_1]=0$ and $\int_{\|x\| >1} \|x\|^2 \nu(dx) < \infty$, $f: S \times \R^+ \to M_{n \times d}(\R)$ a $\mathcal{B}(S \times \R^+)$-measurable function and $f \in L^2(S \times \R^+, \pi \otimes \lambda)$. Then, the resulting causal MMA process $X$ is a $\theta$-weakly dependent process with coefficients
\begin{equation}
\label{coefficients_theta}
\theta_X(r)= \Big( \int_S \int_{-\infty}^{-r} tr(f(A,-s)\Sigma_L f(A,-s)^{\prime}) \, ds \, \pi(dA) \Big)^{\frac{1}{2}}
\end{equation}
for all $r \geq 0$, where  $\E[L_1L_1^{\prime}]=\Sigma_L=\Sigma+\int_{\R^d} x x^{\prime} \nu(dx)$.
\end{Proposition}

\begin{proof}
First, we define $\forall t \in \R$ and $m\geq 0$ the truncated sequence
\begin{equation}
\label{truncated_theta}
X_t^{(m)}=\int_S  \int_{-\infty}^{t} f(A,t-s) \mathbb{1}_{[0,m]}(t-s) \, \Lambda(dA,ds)= \int_S \int_{t-m}^{t} f(A,t-s) \,\, \Lambda(dA,ds).
\end{equation}
Since the kernel function $f$ is square integrable, we have that properties (\ref{ass2}) and (\ref{ass3}) hold and then $f$ is $\Lambda$-integrable (Theorem \ref{uno}) and $X$ is well defined. Thus, because of Proposition \ref{moment1}, $\E[X_t^2] < \infty$ for all $t \in \R$ and we can determine an upper bound of the expectation 
\[
\E\|X_t-X_t^{(m)}\|= \E \Big \| \int_S \int_{-\infty}^{t-m} f(A,t-s) \, \Lambda(dA,ds) \Big \| 
\]
\[
\leq \Big(\E \Big \| \int_S \int_{-\infty}^{t-m} f(A,t-s) \, \Lambda(dA,ds) \Big \|^2 \Big)^{\frac{1}{2}}.
\]
Due to the stationarity of $X$ the above estimation is independent of $t$ and equal to
\begin{equation}
\label{res_theta}
\Big( \int_S\int_{-\infty}^{-m} tr(f(A,-s)\Sigma_L f(A,-s)^{\prime}) \, ds \pi(dA) \Big)^{\frac{1}{2}}
\end{equation}

Let $F$ and $G$ belong respectively to the class of bounded functions $\mathcal{H}^*$ and $\mathcal{H}$, 
$(u,v) \in \mathbb{N}^* \times \mathbb{N}^*$, $r \in \R^{+}$, $(i_1,\ldots,i_u) \in \R^u$ and $(j_1,\ldots,j_v) \in \R^v$ with $i_1\leq\ldots\leq i_u \leq i_u+r\leq j_1\leq \ldots\leq j_v$, $X_i^*=(X_{i_1},\ldots,X_{i_u})$ and $X_j^{*(m)}=(X_{j_1}^{(m)},\ldots,X_{j_v}^{(m)})$ where for all $m\geq 0$
\begin{equation*}
X_{i_u}=\int_S\int_{-\infty}^{i_u} f(A,i_u-s) \, \Lambda(dA,ds)
\end{equation*}
\begin{equation}
\label{ind2}
\textrm{and} \hspace{9.6cm} 
\end{equation}
\begin{equation*}
X_{j_1}^{(m)}=\int_S\int_{j_1-m}^{j_1} f(A,j_1-s) \, \Lambda(dA,ds).
\end{equation*}

Then, if $j_1-m-i_u\geq 0$, which can also be expressed as $j_1-i_u \geq m$, $I_u=S \times (-\infty,i_u]$ and $J_1=S \times [j_1-m,j_1]$ are disjoint sets or they have intersection $S \times \{j_1-m\}$ when $j_1-m=i_u$. Noting that $\pi \times \lambda (S\times \{j_1-m\})=0$, by the definition of a L\'evy basis, the two sequences $(X_i)_{i_1,\ldots, i_u}$ and $(X_j^{(m)})_{j_1,\ldots,j_v}$ are independent and so are $F(X_i^{*})$ and $G(X_j^{*(m)})$.
Therefore, let $m=r$
\[
|Cov(F(X_i^*),G(X_j^*))|\leq |Cov(F(X_i^*), G(X_j^{*})-G(X_j^{*(m)}))|
\]
\[
+ |Cov(F(X_i^*),G(X_j^{*(m)}))| \leq 2 \E |G(X_j^*)-G(X_j^{*(m)})| 
\]
the last relation comes from $\|F\|_{\infty}\leq 1$
\[
\leq 2  Lip(G) \sum_{k=1}^v \E\|X_{j_k}-X_{j_k}^{(m)}\| 
\]
using the result (\ref{res_theta}) 
\[
\leq 2  v Lip(G)  \Big( \int_S\int_{-\infty}^{-r} tr(f(A,-s)\Sigma_L f(A,-s)^{\prime})  \, ds \pi(dA) \Big)^{\frac{1}{2}}
\]
\[
=2 v Lip(G) \,\theta_X(r),
\]
which converges to zero as $r$ goes to infinity by applying the dominated convergence theorem.
\end{proof}

Also in the case of $\theta$-weak dependence, the $\theta$-coefficients change when the underlying L\'evy process has mean different from zero.

\begin{Corollary}
\label{tre_theta_non}
Let $\Lambda$ be an $\R^d$-valued L\'evy basis with characteristic quadruple $(\gamma,\Sigma,\nu,\pi)$ such that $\int_{\|x\| >1} \|x\|^2 \nu(dx) < \infty$, $f: S \times \R^+ \to M_{n \times d}(\R)$ a $\mathcal{B}(S \times \R^+)$-measurable function satisfying assumption (\ref{ass1}) and $f \in L^2(S \times \R^+, \pi \otimes \lambda)$. Then, the resulting causal MMA process $X$ is a $\theta$-weakly dependent process with coefficients
\begin{equation}
\label{coefficients_theta_1}
\begin{array}{ll}
\theta_X(r)=& \Big( \int_S \int_{-\infty}^{-r} tr(f(A,-s)\Sigma_L f(A,-s)^{\prime}) \, ds \, \pi(dA)\\
&+ \| \int_S \int_{-\infty}^{-r} f(A,-s) \mu \, ds \, \pi(dA)\|^2  \Big)^{\frac{1}{2}}
\end{array}
\end{equation}
for all $r \geq 0$, where $\E[L_1]=\mu=\gamma+\int_{\|x\|>1} x \,\nu(dx)$ and $\E[L_1L_1^{\prime}]=\Sigma_L=\Sigma+\int_{\R^d} x x^{\prime} \,\nu(dx)$.
\end{Corollary}

We conclude the study of the $\theta$-weak dependence properties of an MMA process with the computation of the $\theta$-coefficients for an underlying L\'evy process of finite variation.
\begin{Corollary}
\label{quattro_theta}
Let $\Lambda$ be an $\R^d$-valued L\'evy basis with characteristic quadruple $(\gamma,0,\nu,\pi)$ such that $\int_{\R^d} \|x\| \nu(dx) < \infty$, $f: S \times \R^+ \to M_{n \times d}(\R)$ a $\mathcal{B}(S \times \R^+)$-measurable function and $L^1(S \times \R^+, \pi \otimes \lambda)$ and define $\gamma_0$ as in (\ref{yei}).
Then, the resulting causal MMA process $X$ is a $\theta$-weakly dependent process with coefficients
\begin{equation}
\label{coefficients2_theta}
\theta_X(r)=  \int_S \int_{-\infty}^{-r} \int_{\R^d} \| f(A,-s)x\| \,\nu(dx)\, ds \, \pi(dA)
\end{equation}
\[
+\int_S \int_{-\infty}^{-r} \| f(A,-s) \gamma_0 \|  \, ds \, \pi(dA),
\]
for all $r \geq 0$.
\end{Corollary}

\begin{proof}
The kernel function $f$ is in $L^1$ then the properties (\ref{ass4}) and (\ref{ass5}) are satisfied and then $f$ is $\Lambda$-integrable and $X$ well defined and with finite mean by Corollary \ref{moment2}.
Using the notation in Proposition \ref{tre_theta}, for all $t \in \R$ and $m\geq 0$
\begin{align}
\label{trunc_theta}
\begin{split}
\E \| X_t-X_t^{(m)}\|  \leq& \int_S \int_{-\infty}^{-m} \int_{\R^d} \|f(A,-s)x\|  \,\nu(dx)\, ds \, \pi(dA) \\
+& \int_S \int_{-\infty}^{-m} \| f(A,-s) \gamma_0 \| \, ds \, \pi(dA),
\end{split}
\end{align}
where $X_t^{(m)}$ is the truncated sequence (\ref{truncated_theta}).
Thus, for $m=r$ and $F$, $G$, $X_i^*$ and $X_j^*$ 
\begin{align*}
&|Cov(F(X_i^*),G(X_j^*))|\\
&\leq 2 v Lip(G) \Big(\int_S \int_{-\infty}^{-r} \int_{\R^d} \| f(A,-s)x\|  \,\nu(dx)\, ds \, \pi(dA)\\
&+\int_S \int_{-\infty}^{-r} \| f(A,-s) \gamma_0 \|  \, ds \, \pi(dA)\Big)\\
&=2  v Lip(G) \, \theta_X(r).
\end{align*}
Finally, we conclude by applying the dominated convergence theorem.
\end{proof}

\begin{Remark}
The $\eta$-coefficients of a causal MMA process can be chosen to be equal to the $\theta$-coefficients for each $r \geq 0$. This can be easily seen by noticing that the truncated sequences (\ref{ind1}) in Proposition \ref{tre} are equal to the truncated sequences (\ref{ind2}) in Proposition \ref{tre_theta}. This leads to select the parameter $m=r$ in both proofs. Moreover, (\ref{res}) is equal to (\ref{res_theta}) and then $\eta_X(r)=\theta_X(r)$. 
The same observations hold when comparing the results in Corollary \ref{tre_non} or Corollary \ref{quattro} with Corollary \ref{tre_theta_non} or Corollary \ref{quattro_theta}.
\end{Remark}
An example of a causal MMA process is the supOU process studied in \cite{BN01} and \cite{BNS11}.
Let us analyze its weak dependence properties in the univariate case.
We consider the kernel function $f(A,s)= \mathrm{e}^{As} 1_{[0,\infty)}(s)$, $A \in \R^{-}$, $s \in \R$ and $\Lambda$ a $1$-dimensional L\'evy basis on $\R^{-}\times\R$
with generating quadruple $(\gamma,\Sigma,\nu,\pi)$ such that
\begin{equation}
\label{supou_ass}
\int_{|x|>1} \log(|x|) \, \nu(dx) < \infty, \,\, \textrm{and} \,\, \int_{\R^{-}} -\frac{1}{A} \pi(dA) < \infty,
\end{equation}
then the process
\begin{equation}
\label{supou}
 X_t=\int_{\R^{-}} \int_{-\infty}^t \mathrm{e}^{A (t-s)} \,\Lambda(dA,ds)
\end{equation}
is well defined for each $t \in \R$ and strictly stationary. For the supOU process, $A$ represents a random mean reversion parameter.

If $\E[L_1]=0$ and $\int_{|x|>1} |x|^2 \nu(dx) < \infty$, (\ref{supou})  is $\theta$-weakly dependent
with coefficients
\begin{equation}
\label{theta_sup}
\theta_X(r)= \Big( \int_{\R^{-}} \int_{-\infty}^{-r} \mathrm{e}^{-2As} \sigma^2 \, ds \, \pi(dA) \Big)^{\frac{1}{2}}=\Big[ -\sigma^2 \int_{\R^{-}} \frac{\mathrm{e}^{2Ar}}{2A} \, \pi(dA)\Big]^{\frac{1}{2}}
\end{equation}
\[
= Cov(X_0,X_{2r})^{\frac{1}{2}},
\]
by using \cite[Theorem 3.11]{BNS11} and where $\sigma^2= \Sigma+ \int_{\R} x^2 \nu(dx)$.

If $\E[L_1]=\mu$ and $\int_{|x|>1} |x|^2 \nu(dx) < \infty$, the supOU process is $\theta$-weakly dependent
with coefficients
\begin{equation}
\label{theta_sup_1}
\theta_X(r) = \Big(Cov(X_0,X_{2r})+\frac{4\mu^2}{\sigma^4}Cov(X_0,X_r)^2\Big)^{\frac{1}{2}}.
\end{equation}
If $\int_{\R} |x| \nu(dx) < \infty$, $\gamma_0=\gamma-\int_{|x|\leq 1} x \,\nu(dx) >0 $ and $\nu(\R^{-})=0 $, i.e. the underlying L\'evy process is a subordinator, then (\ref{supou}) admits $\theta$-coefficients
\begin{equation}
\label{theta_sup_2}
\theta_X(r) = -\mu \int_{\R^{-}} \frac{\mathrm{e}^{Ar}}{A} \, \pi(dA),
\end{equation}
and when in addition $\int_{|x|>1} |x|^2 \nu(dx) < \infty$
\begin{equation}
\label{theta_sup_3}
\theta_X(r)=\frac{2\mu}{\sigma^2} Cov(X_0,X_r).
\end{equation}

Note that in the finite superposition case strong mixing of the supOU process has been shown in \cite{K14,L12} based on Masuda's result \cite{M07}. As this crucially hinges on an embedding into a finite dimensional Markov process this does not readily extend to the general case.

\begin{Remark}
\label{decay}
The necessary and sufficient condition $\int_{\R^{-}} -\frac{1}{A} \,\pi(dA)$ for the supOU process to exist is satisfied by many continuous and discrete distributions $\pi$, see \cite[Section 2.4]{STW15} for more details. For example, a probability measure $\pi$ being absolutely continuous with density $\pi^{\prime}=(-x)^{\alpha} l(x)$ and regularly varying at zero from the left with $\alpha>0$ (see \cite{BGT87}), i.e. l is slowly varying at zero, satisfies the above condition. If moreover, $l(x)$ is continuous in $(-\infty,0)$ and $\lim_{x \to 0^{-}} l(x) >0$ exists, it holds that 
\[
Cov(X_0,X_r)\sim \frac{C}{r^{\alpha}}, \,\,\textrm{with a constant $C>0$ and $r \in \R^+$} 
\]
where for $\alpha \in (0,1)$ the supOU process exhibits long memory and for $\alpha > 1$ short memory, see \cite[Definition 3.1.2]{G12}. Concrete examples where the covariances are calculated explicitely, in this set-up, can be found in \cite{BNL05}.
\end{Remark}
\begin{Remark}
A natural question is whether one can improve the weak dependence coefficients that we obtain. 

		\cite[Lemma 4.1]{DW07} shows that for stationary processes with finite $m$-moments ($m>2+\delta$, for $\delta>0$) being $\lambda$-weakly dependent (cf. \cite[Definition 2.1]{DW07}) and thus $\eta$-weakly dependent 
		\[
		|Cov(X_0,X_r)|\leq 9 \, \E[\|X_0\|^m]^{\frac{1}{m-1}} \, \lambda(r)^{\frac{m-2}{m-1}}.
		\]
		The above arguments can be easily adapted to the causal case and  $\theta$-weak dependence where we likewise get
		\[
		|Cov(X_0,X_r)|\leq 9 \, \E[\|X_0\|^m]^{\frac{1}{m-1}} \, \theta(r)^{\frac{m-2}{m-1}}.
		\]
		If the stationary process has finite moments of any order we thus obtain the inequalities
		\begin{equation}\label{covbound}
		|Cov(X_0,X_r)|\leq 9 \, \eta(r)\,\,\,\, \textrm{and}\,\,\,\, |Cov(X_0,X_r)|\leq 9\,\theta(r).
		\end{equation}
		
		Equation \eqref{theta_sup_3} shows that our weak dependence coefficients are sharp for a supOU process having as underlying L\'evy process a subordinator with finite second moment. Note that ``sharp'' here means that the right and left hand side of the inequalities \eqref{covbound} only differ by a constant, as for the weak dependence coefficients one usually - like in the upcoming CLTs -  only cares about their summability/integrability in $r$. The inequalities \eqref{covbound} compared to \eqref{theta_sup} and \eqref{theta_sup_1} show that we might obtain smaller weak dependence coefficients for the supOU process having an underlying L\'evy process of infinite variation. In fact following Remark \ref{decay}, if $Cov(X_0,X_r)\sim r^{-\alpha}$ for $\alpha >0$, then the left hand side in \eqref{covbound} decays like $r^{-\alpha}$ whereas the right hand side decays like $r^{-\alpha/2} $.
		
		Inspecting the proofs of Corollaries \ref{quattro} and \ref{quattro_theta} and Propositions \ref{tre} and \ref{tre_theta}, where the $\eta$ and $\theta$-coefficients are determined, the crucial issue is that we use the equality (\ref{yei2}) to compute a bound of the term $\E\|X_t-X_t^{(m)}\|$  in the finite variation case, whereas we bound $\E\|X_t-X_t^{(m)}\|$ by means of a second moment in the infinite variation one. We do the latter because to the best of our knowledge there are no sharper bounds known for the first absolute moment of an infinitely divisible distribution that are suitably expressible in terms of the characteristic triplet in this set-up.
	\end{Remark}

To conclude, we state the hereditary property of a $\theta$-weakly dependent process.
\begin{Proposition}
\label{her3}
Let $(X_t)_{t \in \R}$ be an $\R^n$-valued stationary process and assume there exists some constant $C>0$ such that $ \E[\|X_{0}\|^p]^{\frac{1}{p}} \leq C,$ with $p>1$, $h \colon \R^n \to \R^m$ be a function such that $h(0)=0$, $h(x)=(h_1(x),\ldots,h_m(x))$ 
and 
\begin{equation*}
\|h(x)-h(y)\| \leq c \,\|x-y\| (1+\|x\|^{a-1}+\|y\|^{a-1}),
\end{equation*}
for $x,y \in \R^n$, $c>0$ and $1\leq a < p$.
Define $(Y_t)_{t \in \R}$ by $Y_t=h(X_t)$. If  $(X_t)_{t \in \R}$ is a $\theta$-weakly dependent process,
then $(Y_t)_{t \in \R}$ is a $\theta$-weakly dependent process such that
\[
\forall \, r \geq 0, \,\,\, \theta_Y(r)= \mathcal{C} \, \theta_X(r)^{\frac{p-a}{p-1}},
\]
with the constant $\mathcal{C}$ independent of $r$. 
\end{Proposition}

\begin{proof}
Analogous to Proposition \ref{her2}.
\end{proof}

\section{Sample moments of an MMA process}
\label{sec3}

We consider a sample of $N$ observations of a univariate MMA process $\{X_{\Delta},\ldots,X_{N\Delta}\}$, where $\Delta$ is a positive integer.
\begin{equation}
\label{timeseries}
X_{i\Delta}\colon= \int_S \int_{\R} f(A,i\Delta-s) \Lambda(dA,ds),\,\,\, \textrm{for $i \in \Z$.}
\end{equation}
If the underlying L\'evy process $L$ has finite first moment, we define $\tilde{X}_{i\Delta}=X_{i\Delta}-\E[X_0]$. 

The sample mean of the process $X$ is defined as
\begin{equation}
\label{mean}
 \frac{1}{N} \sum_{i=1}^N X_{i\Delta} 
\end{equation}
and its sample autocovariance function at lag $k \in \N$.
\begin{equation}
\label{cova}
\frac{1}{N}\sum_{j=1}^{N} \tilde{X}_{j\Delta}\tilde{X}_{(j+k)\Delta}.
\end{equation}

W.l.o.g, we consider below $\E[X_0]=0$ and $\Delta=1$ in order to lighten the notations and, 
when the asymptotic properties of the sample auto-covariance functions are investigated, we focus on the features of the 
processes 
\begin{equation}
\label{timeseries2}
Y_{j,k}=X_{j} X_{j+k}-D(k) \,\,\, \textrm{for all} \,\,\,  k\in \N,
\end{equation} 
where we denote by $D(k)$ the covariances at lag $k$ defined, when $\E[X_0]=0$, by 
\begin{equation}
\label{passo}
 D(k)= Cov(X_0,X_{k})=\E[X_0X_k]=\int_S\int_{\R} f(A,-s) \Sigma_L f(A,k-s)^{\prime} \,ds \, \pi(dA),
\end{equation}
for $k \in \Z$, where $\E[L_1 L_1^{\prime}]=\Sigma_L=\Sigma+\int_{\R^d} x x^{\prime} \nu(dx)$.

We start by analyzing the asymptotic properties of the sample mean (\ref{mean}) for a non-causal and a causal MMA process.

\begin{Theorem}
\label{clt_mean}
Let $\Lambda$ be an $\R^d$-valued L\'evy basis with characteristic quadruple $(\gamma,\Sigma,\nu,\pi)$ such that $\E[L_1]=0$ and $\int_{\|x\| >1} \|x\|^{2+\delta} \nu(dx) < \infty$, for some $\delta>0 $, $f: S \times \R \to M_{1 \times d}(\R)$ a $\mathcal{B}(S \times \R)$-measurable function and $f \in L^{2+\delta}(S \times \R,\pi \otimes \lambda)\cap L^{2}(S \times \R, \pi \otimes \lambda)$.
If $(X_i)_{i \in \Z}$ as defined in (\ref{timeseries}) is an $\eta$-weakly dependent process with coefficients $\eta_X(r)= O(r^{-\beta})$ and $\beta>4+\frac{2}{\delta}$,
then
\begin{equation}
\label{var1}
\sigma^2_{\eta} = \sum_{k \in \Z} Cov(X_0,X_k)
\end{equation}
is finite, non-negative and as $N \to \infty$
\begin{equation}
\label{thesis}
\frac{1}{\sqrt{N}}\sum_{i=1}^N X_i \stackrel{d}{\rightarrow} \mathcal{N}(0,\sigma^2_{\eta}).
\end{equation}
\end{Theorem}

\begin{proof}
We have that $\E[X_0^{2+\delta}]<\infty$ for $\delta>0$ because of Proposition \ref{moment1}. Moreover, the $\eta$-weakly dependent process $X$ satisfies the sufficient conditions of \cite[Theorem 2.2]{DW07}. The absolute summability of the series (\ref{var1}) follows and so the asymptotic normality (\ref{thesis}).
\end{proof}

In the case of a causal MMA process, the required decay rate of the $\theta$ coefficients is lower than in the $\eta$-weak dependence case. 
\begin{Theorem}
\label{clt_mean2}
Let $\Lambda$ be an $\R^d$-valued L\'evy basis with characteristic quadruple $(\gamma,\Sigma,\nu,\pi)$ such that $\E[L_1]=0$ and $\int_{\|x\| >1} \|x\|^{2+\delta} \nu(dx) < \infty$, for some $\delta>0 $, $f: S \times \R^+ \to M_{1 \times d}(\R)$ a $\mathcal{B}(S \times \R^+)$-measurable function and $f \in L^{2+\delta}(S \times \R^+,\pi \otimes \lambda)\cap L^{2}(S \times \R^+, \pi \otimes \lambda)$.
If $(X_i)_{i \in \Z}$ as defined in (\ref{timeseries}) is a $\theta$-weakly dependent process with coefficients $\theta_X(r)= O(r^{-\alpha})$ and $\alpha>1+\frac{1}{\delta}$,
then 
\begin{equation}
\label{var2}
\sigma^2_{\theta} = \sum_{k \in \Z} Cov(X_0,X_k)
\end{equation}
is finite, non-negative and as $N \to \infty$
\begin{equation}
\label{thesis2}
\frac{1}{\sqrt{N}}\sum_{i=1}^N X_i \stackrel{d}{\rightarrow} \mathcal{N}(0,\sigma^2_{\theta}).
\end{equation}
\end{Theorem}

\begin{proof}
The MMA process has finite $2+\delta$-moment for $\delta>0$ (Proposition \ref{moment1}) and is ergodic, as shown in \cite{FS13}. By Lemma 2 in \cite{DD03}, the condition $D(2,\theta /2, X_0)$ holds. Then, by Corollary 1 in \cite{DD03} and the ergodicity of the process $X$, (\ref{thesis2}) follows by applying \cite[Theorem 1]{DR00}.
\end{proof}

\begin{Remark}
In the special case of the supOU process, being representable as a finite sum of independent Ornstein-Uhlenbeck processes with gamma or inverse gaussian marginals, a comparable result can be found in \cite[Theorem 2]{L12}. 
\end{Remark}

\begin{Remark}
Theorem \ref{clt_mean} and \ref{clt_mean2} as well as all the upcoming central limit theorems can be also formulated as functional central limit theorems, following \cite{DW07} and \cite{DD03} respectively. However, we state the theorems just for the sample moments we are interested in (and which we are using in Section \ref{sec5}) to lighten the notations.

For example, denote for $t \in[0,1]$ and $n\geq1$
\[
S_n(t)= X_1+\cdots+X_{[nt]}, \,\,\, 
\]
in the case of a non-causal MMA process, and
\[
S^*_n(t)=X_1+\cdots+X_{[nt]}+(nt-[nt]) X_{[nt]+1}
\]
for a causal MMA process. Then, under the assumptions of Theorem \ref{clt_mean} or \ref{clt_mean2}, $n^{-\frac{1}{2}} S_n(t)$ converges in distribution in the Skorohod space $D[0,1]$ to $\sigma_{\eta}W$and $n^{-\frac{1}{2}} S^*_n(t)$ converges in distribution in the space $C[0,1]$ to $\sigma_{\theta}W$, respectively. Here, $W$ denotes a standard Brownian motion.
\end{Remark}

\begin{Remark}
\label{clt_fv}
In the finite variation case, (\ref{thesis}) or (\ref{thesis2}), respectively, hold under \\
$\int_{\|x\| >1} \|x\|^{2+\delta} \nu(dx) < \infty$, for some $\delta >0$, and $f \in L^{2+\delta}(S \times \R,\pi \otimes \lambda)\cap L^{1}(S \times \R, \pi \otimes \lambda)$ or $f \in L^{2+\delta}(S \times \R^+,\pi \otimes \lambda)\cap L^{1}(S \times \R^+, \pi \otimes \lambda)$, respectively.
\end{Remark}

\begin{Remark}
\label{non_degenerate}
Assuming that $f$ is not equal to zero $\pi$-almost everywhere and that $f \geq 0$ or $f \leq 0$, the asymptotic variance in Theorems \ref{clt_mean} and \ref{clt_mean2} is not degenerate.
This is the case for example when working with the supOU process (\ref{supou}). Moreover, it is worthy to observe that the assumptions in Theorem \ref{clt_mean2} clearly indicate that we obtain asymptotic normality of the sample mean for a causal MMA process just in the short memory case.
\end{Remark}

To find an asymptotic distribution for the sample autocovariance functions (\ref{cova}), we first show that $(Y_{j,k})_{j \in \Z}$ are $\eta$-weakly or $\theta$-weakly dependent processes. 
In addition to the hereditary properties in Proposition \ref{her2} and \ref{her3}, we need to establish when the weak dependence properties of 
a univariate MMA process are inherited by the process $Z_t=(X_t,X_{t+1},\ldots,X_{t+k})$ for all $k \in \N$.

\begin{Proposition}
\label{lem}
Let $\Lambda$ be an $\R^d$-valued L\'evy basis with generating quadruple $(\gamma,\Sigma,\nu,\pi)$ and $f:S \times \R \to M_{1 \times d}(\R)$ be a 
$\mathcal{B}(S \times \R)$-measurable function satisfying the assumptions of Theorem \ref{uno}. If for all $t \in \R$, $X$ is a non-causal or causal MMA as defined in (\ref{mma}) or (\ref{mma_causal}) respectively, then\[
Z_t=\int_S \int_{\R} g(A,t-s) \, \Lambda(dA,ds),
\]
where $g(A,s)=\left(
  \begin{array}{l}
\,\,\,\,f(A,s)\\
f(A,s-1)\\
\,\,\,\,\,\,\,\,\ldots\\
f(A,s-k)
\end{array}
\right)$ is a $\mathcal{B}(S \times \R)$-measurable function with values in $M_{k+1 \times d}(\R)$ and $k\in \N$, is an MMA process.
Moreover, if $X$ satisfies the assumptions of Proposition \ref{tre} (Corollary \ref{tre_non}) or Proposition \ref{tre_theta} (Corollary \ref{tre_theta_non}) then $Z$ is $\eta$ or $\theta-$weakly dependent respectively with coefficients
\begin{equation}
\label{coeff_lag}
\eta_Z(r)= \mathcal{D} \, \eta_X(r-2k) \,\textrm{for $r \geq 2k$}\,\,\,\textrm{or}\,\,\, \theta_Z(r)= \mathcal{D} \, \theta_X(r-k) \,\textrm{for $r \geq k$},
\end{equation}
where $\mathcal{D}=(k+1)^{\frac{1}{2}}$.
In the case when the assumptions of Corollaries \ref{quattro} or \ref{quattro_theta} hold, the process $Z$ is respectively $\eta$ or $\theta$-weakly dependent with coefficients
\begin{equation}
\label{coeff_lag_finite}
\eta_Z(r)= \mathcal{D} \,\eta_X(r-2k) \,\textrm{for $r \geq 2k$}\,\,\,\textrm{or}\,\,\, \theta_Z(r)= \mathcal{D} \, \theta_X(r-k) \,\textrm{for $r \geq k$},
\end{equation} 
and $\mathcal{D}=k+1$. 
\end{Proposition}

\begin{proof}
\noindent
For $k=1$, the first step of the proof consists of checking that $g$ is a $\Lambda$-integrable function as prescribed by Theorem \ref{uno}.
W.l.o.g, we consider in our calculations the norm
\[
\|(x,y)\|=\|x\|+\|y\|
\]
for $x, y \in M_{1 \times d}(\R)$.
We have that 

\begin{equation}
\label{vec1}
\int_S \int_{\R} \Big\|g(A,s)\gamma+ \int_{\R^d} g(A,s) x \Big(\mathbb{1}_{[0,1]}(\|g(A,s)x\|)-\mathbb{1}_{[0,1]}(\|x\|)\Big) \,\, \nu(dx) \Big\| \, ds  \, \pi(dA) 
\end{equation}
\[
=\int_S \int_{\R} \Big\| \left(
  \begin{array}{l}
\,\,\,\,f(A,s)\\
f(A,s-1)
\end{array}
\right)\gamma 
\]
{\small
\[
+ \int_{\R^d} \left(
  \begin{array}{l}
\,\,\,\,f(A,s)\\
f(A,s-1)
\end{array}
\right) x \Big(\mathbb{1}_{[0,1]}\Big(\Big\| \left(\begin{array}{l}
\,\,\,\,f(A,s)\\
f(A,s-1)
\end{array}
\right)x \Big\|\Big)-\mathbb{1}_{[0,1]}(\|x\|)\Big) \nu(dx) \Big\| \,\ ds  \, \pi(dA)
\]
\[
= \int_S \int_{\R} \Big\|f(A,s)\gamma+ \int_{\R^d} f(A,s) x \Big(\mathbb{1}_{[0,1]}(\|g(A,s)x\|)-\mathbb{1}_{[0,1]}(\|x\|)\Big) \,\, \nu(dx) \Big\| \, ds  \, \pi(dA)
\]
\[
+\int_S \int_{\R} \Big\|f(A,s-1)\gamma+ \int_{\R^d} f(A,s-1) x \Big(\mathbb{1}_{[0,1]}(\|g(A,s)x\|)-\mathbb{1}_{[0,1]}(\|x\|)\Big) \,\, \nu(dx) \Big\| \, ds  \, \pi(dA).
\]}
Noting that 
\begin{equation*}
\mathbb{1}_{[0,1]}(\|g(A,s)x\|) \leq \mathbb{1}_{[0,1]}(\|f(A,s)x\|)
\end{equation*}
\begin{equation*} 
\textrm{and} \hspace{12.6cm} 
\end{equation*}
\begin{equation*}
\mathbb{1}_{[0,1]}(\|g(A,s)x\|) \leq \mathbb{1}_{[0,1]}(\|f(A,s-1)x\|),
\end{equation*}
it then holds that (\ref{vec1}) is finite.

Let us pass to the second condition 
\[
\int_S \int_{\R} \|g(A,s)\Sigma g(A,s)^{\prime}\| \, ds \, \pi(dA)
\]
\[
\leq \int_S \int_{\R} \Big(\|f(A,s) \Sigma f(A,s)^{\prime} \| + \|f(A,s-1) \Sigma f(A,s-1)^{\prime} \| \Big)\,\, ds \,\pi(dA).
\]
Therefore, $f$ being a kernel of an MMA process, the above integral is finite.
Finally, we have 
\[
\int_S \int_{\R} \int_{\R^d} \Big(1 \wedge \|g(A,s)x\|^2 \Big) \, \nu(dx) \, ds \, \pi(dA) 
\]
\[
\leq 2 \int_S \int_{\R} \int_{\R^d} \Big(1 \wedge \|f(A,s)x\|^2 \Big) \, \nu(dx) \, ds \, \pi(dA) 
\]
\[
+ 2 \int_S \int_{\R} \int_{\R^d} \Big(1 \wedge \|f(A,s-1)x\|^2 \Big) \, \nu(dx) \, ds \, \pi(dA) < \infty.
\]
Thus the kernel function $g$ is a $\Lambda$-integrable function. By induction the statement can be shown for each $k \in \N$.
Because all the assumptions of Theorem \ref{uno} hold, we have that $Z$ is an MMA process.

Depending now on the properties of the underlying L\'evy process, we can distinguish three
different scenarios for the $\eta$ and $\theta$-weak dependence. 
When $X$ satisfies the assumptions of Proposition \ref{tre},
\[
\eta_Z(r) = \Big( \int_S\int_{\R} tr(g(A,-s)\Sigma g(A,-s)^{\prime}) \, \mathbb{1}_{(-\infty, -r) }(2s) \, ds \pi(dA) \Big)^{\frac{1}{2}}
\]
\[
+ \Big( \int_S\int_{\R} tr(g(A,-s)\Sigma g(A,-s)^{\prime}) \, \mathbb{1}_{(r, +\infty) }(2s) \, ds \pi(dA) \Big)^{\frac{1}{2}}
\]
$$
\leq \Big( \int_S\int_{\R} tr(f(A,-s)\Sigma f(A,-s)^{\prime}) \, \mathbb{1}_{(-\infty, -r) }(2s) \, ds \pi(dA)+\ldots
$$
\[
+ \int_S\int_{\R} tr(f(A,k-s)\Sigma f(A,k-s)^{\prime}) \, \mathbb{1}_{(-\infty, -r) }(2s) \, ds \pi(dA) \Big)^{\frac{1}{2}}
\]
\[
+ \Big( \int_S\int_{\R} tr(f(A,-s)\Sigma f(A,-s)^{\prime}) \, \mathbb{1}_{(r, +\infty) }(2s) \, ds \pi(dA) + \ldots
\]
\[
+\int_S\int_{\R} tr(f(A,k-s)\Sigma f(A,k-s)^{\prime}) \, \mathbb{1}_{(r, +\infty) }(2s) \, ds \pi(dA) \Big)^{\frac{1}{2}}
\]
\[
\leq (k+1)^{\frac{1}{2}} \Big( \int_S\int_{\R} tr(f(A,-s)\Sigma f(A,-s)^{\prime}) \, \mathbb{1}_{(-\infty, -r +2k) }(2s) \, ds \pi(dA) \Big)^{\frac{1}{2}}
\]
\[
+ \Big( \int_S\int_{\R} tr(f(A,-s)\Sigma f(A,-s)^{\prime}) \, \mathbb{1}_{(r-2k, +\infty) }(2s) \, ds \pi(dA) \Big)^{\frac{1}{2}}
\]
\[
\leq (k+1)^{\frac{1}{2}} \eta_X(r-2k),
\]

for each $r > 2k$.
Thus, $Z$ is a $k+1$-dimensional MMA process with $\eta$ coefficients
\[
\eta_Z(r)=(k+1)^{\frac{1}{2}} \eta_X(r-2k).
\]
If $X$ satisfies the assumptions of Proposition \ref{tre_theta}, it can be shown similarly that $Z$ is $\theta$-weakly dependent with coefficients\[
\theta_Z(r)=(k+1)^{\frac{1}{2}} \theta_X(r-k).
\]

Similar calculations follow in the finite variation case leading to the statements (\ref{coeff_lag_finite}).
\end{proof}

\begin{Proposition}
\label{eta2}
Let $\Lambda$ be an $\R^d$-valued L\'evy basis with characteristic quadruple 
$(\gamma,\Sigma,\nu,\pi)$ and $\int_{\|x\| >1} \|x\|^{2+\delta} \nu(dx) < \infty$,
for some $\delta>0 $, $f: S \times \R \to M_{1 \times d}(\R)$ a $\mathcal{B}(S \times \R)$-measurable 
function and $f \in L^{2+\delta}(S \times \R,\pi \otimes \lambda)\cap L^{2}(S \times \R, \pi \otimes \lambda)$. 
If $(X_i)_{i \in \Z}$ as defined in (\ref{timeseries}) is $\eta$ or $\theta$-weakly dependent respectively with coefficients $\eta_X$ or $\theta_X$,
then for all $k \geq 0$ the processes $(Y_{j,k})_{j \in \Z}$ are respectively $\eta$ or $\theta$-weakly dependent with coefficients
\[
\eta_Y(r)= \mathcal{C} (\sqrt{2} \eta_X(r-2k))^{\frac{\delta}{1+\delta}}\,\,\,\textrm{or}\,\,\,\theta_Y(r)= \mathcal{C} (\sqrt{2} \theta_X(r-k))^{\frac{\delta}{1+\delta}}.
\]
If $L$ is a process of finite variation and $f \in L^{2+\delta}(S \times \R,\pi \otimes \lambda)\cap L^{1}(S \times \R, \pi \otimes \lambda)$, then
\[
\eta_Y(r)= \mathcal{C} (2 \eta_X(r-2k))^{\frac{\delta}{1+\delta}}\,\,\,\textrm{or}\,\,\,\theta_Y(r)= \mathcal{C} (2 \theta_X(r-k))^{\frac{\delta}{1+\delta}}.
\]
The constant $\mathcal{C}$, appearing in the $\eta$ and $\theta$-coefficients, is independent of $r$. 
\end{Proposition}

\begin{proof}
Let us consider a $2$-dimensional process $Z=(X_j,X_{j+k})_{j \in \Z}$ with $k \in \N$.
The $\eta$ or $\theta$ coefficients of the process $Z$ are
\[
\eta_Z(r)=\sqrt{2} \eta_X(r-2k) \,\,\,\textrm{or}\,\,\, \theta_Z(r)=\sqrt{2} \theta_X(r-k)
\]
by Proposition \ref{lem}. The $2+\delta$ moment, for $\delta>0$, of the MMA process exists because of Proposition \ref{moment1}. 
Let us now consider $h \colon \R^2 \to R$ as $h(x_1,x_2)=x_1x_2$. The function $h$ satisfies 
assumption (\ref{hyp_her2}), for $p=2+\delta$, $c=1$ and $a=2$. Then, Proposition \ref{her2} or \ref{her3} applies and $h(Z)=X_j X_{j+k}$, 
as well as $Y_{j,k}$, has either coefficients
\[
\eta_Y(r)= \mathcal{C} (\sqrt{2} \eta_X(r-2k))^{\frac{\delta}{1+\delta}}\,\,\, \textrm{or}\,\,\,\theta_Y(r)= \mathcal{C} (\sqrt{2} \theta_X(r-k))^{\frac{\delta}{1+\delta}}.
\]
The finite variation case follows easily by applying Proposition \ref{lem} and using the coefficients (\ref{coeff_lag_finite}).
\end{proof}

We can now give a distributional limit theorem for the processes $Y_{j,k}$, namely determining the asymptotic distribution of 
\[
\frac{1}{N}\sum_{j=1}^{N} Y_{j,k}\,\,\,\, \textrm{for all $k \in \N$}.
\]

\begin{Corollary}
\label{clt_cova}
Let $\Lambda$ be an $\R^d$-valued L\'evy basis with characteristic quadruple $(\gamma,\Sigma,\nu,\pi)$ such that $\E[L_1]=0$, $\int_{\|x\| >1} \|x\|^{4+\delta} \nu(dx) < \infty$, for some $\delta >0$, $f: S \times \R \to M_{1 \times d}(\R)$ a $\mathcal{B}(S \times \R)$-measurable function and $f \in L^{4+\delta}(S \times \R, \pi \otimes \lambda)\cap L^{2}(S \times \R, \pi \otimes \lambda)$. Let  $(Y_{j,k})_{j \in \Z}$ be defined as in (\ref{timeseries2}) for $k \in \N$. 
If $(X_i)_{i \in \Z}$ as defined in (\ref{timeseries}) is $\eta$-weakly dependent with coefficients $\eta_X(r)= O(r^{-\beta})$ such that $\beta>(4+\frac{2}{\delta})(\frac{3+\delta}{2+\delta})$ or it is $\theta$-weakly dependent with coefficients $\theta_X(r)=O(r^{-\alpha})$ such that $ \alpha> (1+\frac{1}{\delta})(\frac{3+\delta}{2+\delta}) $, then
 $$\gamma_k^2=\sum_{l \in \Z} Cov(Y_{0,k},Y_{l,k})= \sum_{l \in \Z} Cov(X_0X_k,X_lX_{l+k})$$ 
is finite, non-negative and as $N \to \infty$
\begin{equation}
\label{asy_ciao}
\frac{1}{\sqrt{N}}\sum_{j=1}^{N} Y_{j,k} \stackrel{d}{\rightarrow} \mathcal{N}(0,\gamma_k^2)
\end{equation}
\end{Corollary}

\begin{proof}
Since the results of Proposition \ref{moment1} apply, by using \cite[Theorem 2.2]{DW07} in the case of an $\eta$-weakly dependent process or \cite[Theorem 1]{DR00} when the process $X$ is $\theta$-weakly dependent, the distributional limit (\ref{asy_ciao}) holds.
\end{proof}

\begin{Remark}
The asymptotic variance $\gamma^2_k$ can be expressed in terms of the fourth order cumulant of a zero mean MMA process and its covariances as follows. Let us consider an $\R^4$-valued MMA process $X=(X_i,X_j,X_k,X_l)$ with $(i,j,k,l) \in \R^4$ and kernel function $g(A,s)=[f(A,s-i),f(A,s-j),f(A,s-k),f(A,s-l)]^{\prime}$ with values in $M_{4 \times d}(\R)$. The L\'evy basis $\Lambda$, underlying the definition of $X$, satisfies the assumptions of Corollary \ref{clt_cova}. Thus, $X$ is infinitely divisible with characteristic triplet $(\gamma_{int}, \Sigma_{int}, \nu_{int})$ as given in Theorem \ref{uno} and characteristic exponent 
\[
log(\E[\mathrm{e}^{\mathrm{i}\langle u, X \rangle}])=\mathrm{i} \langle \gamma_{int}, u \rangle - \frac{1}{2}\langle u, \Sigma_{int} u \rangle + \int_{\R^d} \mathrm{e}^{\mathrm{i}\langle u,x \rangle}-1-\mathrm{i}\langle u,x \rangle \mathbb{1}_{[0,1]}(\|x\|) \,\,\nu_{int}(dx).
\]
We denote by $\kappa(i,j,k,l)$ the fourth order cumulant of X. By \cite[Proposition 4.2.2.]{G12}
\[
\kappa(i,j,k,l)=\E[X_iX_jX_kX_l]-\E[X_iX_j]\E[X_kX_l]-\E[X_iX_k]\E[X_jX_l]-\E[X_iX_l]\E[X_jX_k].
\]
On the other hand, cf. \cite[Definition 4.2.1]{G12},
\[
\kappa(i,j,k,l)=\frac{\partial^4}{\partial u_1 \partial u_2 \partial u_3 \partial u_4} log(\E[\mathrm{e}^{\mathrm{i}\langle u, X \rangle}]) \Big|_{u_1=u_2=u_3=u_4=0}
\]
\begin{equation}
\label{cumulant}
=  \int_S \int_{\R} \int_{\R^d} f(A,s-i)\, x \, x^{\prime} \, f(A,s-j)^{\prime} \, f(A,s-k) \, x \, x^{\prime} \, f(A,s-l)^{\prime} \, \nu(dx) \, ds \, \pi(dA).
\end{equation}

Then, $\forall (l,k) \in \R^2$
\[
Cov(Y_{0,k},Y_{l,k})= \kappa(0,k,l,l+k) + D(l)^2+ D(k+l)D(k-l).
\]
where $D(l)$ is defined in (\ref{passo}). 

Analogously, the formula to compute the third order cumulant $\kappa(i,j,k)$ can be derived. In fact, 
\begin{equation}
\label{cumulant3}
\kappa(i,j,k) = \int_S \int_{\R} \int_{\R^d} f(A,s-i)\, x \, x^{\prime} \, f(A,s-j)^{\prime} \, f(A,s-k) \, x \, \nu(dx) \, ds \, \pi(dA)
\end{equation}
and for a zero mean MMA process holds that $\E[X_iX_jX_k]=\kappa(i,j,k)$. 
This computation is useful in Section \ref{sec5}.
\end{Remark}

\begin{Corollary}
\label{clt_multi_acf}
Let $\Lambda$ be an $\R^d$-valued L\'evy basis with characteristic quadruple $(\gamma,\Sigma,\nu,\pi)$ such that $\E[L_1]=0$, $\int_{\|x\| >1} \|x\|^{4+\delta} \nu(dx) < \infty$, for some $\delta >0$, $f: S \times \R \to M_{1 \times d}(\R)$ a $\mathcal{B}(S \times \R)$-measurable function and $f \in L^{4+\delta}(S \times \R, \pi \otimes \lambda)\cap L^{2}(S \times \R, \pi \otimes \lambda)$. Let  $\mathcal{Z}_j=(Y_{j,0},\ldots,Y_{j,k})$ for all $j \in \Z$. If $(X_i)_{i \in \Z}$ is $\eta$-weakly dependent with coefficients $\eta_X(r)= O(r^{-\beta})$ such that $\beta>(4+\frac{2}{\delta})(\frac{3+\delta}{2+\delta})$ or it is $\theta$-weakly dependent with coefficients $\theta_X(r)=O(r^{-\alpha})$ such that $ \alpha> (1+\frac{1}{\delta})(\frac{3+\delta}{2+\delta}) $, then respectively for each $p,q \in \{0,\ldots,k\}$ with $k \in \N$,
\[
\sum_{l \in \Z} Cov(X_0X_p, X_l X_{l+q}) = \sum_{l \in \Z} k(0,p,l,l+q)+D(l)D(l+q-p)+D(q+l)D(l-p),
\]
where $k(l,i,j,k)$ is defined in (\ref{cumulant}) and $D(k)$ in (\ref{passo}) for each $l,i,j,k \in \Z$,
is finite and as $N \to \infty$
\[
\frac{1}{\sqrt{N}}\sum_{j=1}^{N} \mathcal{Z}_j \stackrel{d}{\rightarrow} \mathcal{N}_{k+1}(0,\Xi)
\]
where $\Xi$ is equal to 
\begin{equation*}
\small
\left[\begin{array}{cccc}
\sum_{l \in \Z} Cov(X_0^2,X_l^2)& \sum_{l\in \Z} Cov(X_0^2,X_lX_{l+1})& \ldots & \sum_{l \in \Z} Cov(X_0^2,X_lX_{l+k})\\
\ldots &\sum_{l \in \Z} Cov(X_0X_1,X_l X_{l+1}) & \ldots & \sum_{l \in \Z} Cov(X_0X_1,X_l X_{l+k})\\
\ldots& \ldots &\ldots & \sum_{l \in \Z} Cov(X_0X_{k},X_l X_{l+k})
\end{array}\right]
\end{equation*}
and positive semidefinite.
\end{Corollary}

\begin{proof}
Let us consider the vector $Z$ as defined in Proposition \ref{lem}. Given the assumptions of the Corollary, $Z$ is $\eta$-weakly dependent with coefficients $\eta_Z(r)=\mathcal{D}\eta_X(r-2k)$ or $\theta$-weakly dependent with coefficients $\theta_Z(r)=\mathcal{D}\theta_X(r-k)$ because of Proposition \ref{lem} and given that the results of Proposition \ref{moment1} apply.
We apply now the function $f:\R^{k+1}\to \R^{k+1}$ to the vector $Z$ such that
\[
f(Z_j)=\left(
  \begin{array}{l}
\,\,\,\,X_j^2\\
\,\,\,\,\,\,\vdots\\
\,\,\,\,X_jX_{j+k}\\
\end{array}
\right)= \mathcal{Z}_j + \left(\begin{array}{l}
\,\,\,\,D(0)\\
\,\,\,\,\,\,\vdots\\
\,\,\,\,D(k)\\
\end{array}\right).
\]
The assumptions of Proposition \ref{her2} hold with $p=4+\delta$, $c=1$, $a=2$, then $f(Z_t)$ is $\eta$ or $\theta$-weakly dependent with coefficients $\mathcal{C}(\mathcal{D}\eta_X(r-2k))^{\frac{2+\delta}{3+\delta}}$ or $\mathcal{C}(\mathcal{D}\theta_X(r-k))^{\frac{2+\delta}{3+\delta}}$. 
The process $\mathcal{Z}$ is then a process with the same weak dependence coefficients as $f(Z_t)$.
For all $a \in \R^{k+1}$, $a^{\prime}\mathcal{Z}$ is an $\eta$ or a $\theta$-weakly dependent process, 
because a linear function is Lipschitz, having the same coefficients as the process $\mathcal{Z}$.
By \cite[Theorem 2.2]{DW07} or \cite[Theorem 1]{DR00}, then
\[
\frac{1}{\sqrt{N}} \sum_{j=1}^N a^{\prime}\mathcal{Z}_j \xrightarrow{d} \mathcal{N}(0, a^{\prime} \Sigma a) 
\]
as $N \to \infty$.
Applying the Cramer-Wold device, the asymptotic normality of the vector $\mathcal{Z}$ is shown.
\end{proof}

\begin{Remark}
In this paper we consider the classical case of equidistant observations. In many applications different sampling schemes are also highly relevant and for some special cases results have been obtained. For example,  \cite{BC18} considers the asymptotics of the autocovariance function for L\'evy-driven moving average processes sampled at an independent renewal sequence and \cite{FS18} consider the asymptotics of the pathwise Fourier transform/periodogram for L\'evy-driven CARMA processes sampled at deterministic irregular grids. Considering independent renewal sampled L\'evy-driven MMA processes is beyond the scope of the present paper and the content of future research just starting in  \cite{BCS19} where the preservation of strong mixing and weak dependence properties is discussed in general. 
\end{Remark}

\section{Sample moments of an MMA SV model}
\label{sec4}

Let us consider a L\'evy basis with characteristic quadruple $(\gamma, \sigma^2, \nu, \pi)$ and values in $\R$ and the respective univariate casual MMA process $X$ with kernel function $f: (S \times \R^+)\to \R$. Its dependence structure is given by
\[
Cov(X_0,X_t)=\int_S \int_{-\infty}^0 f(A,-s) \sigma^2 f(A,t-s) \, ds \, d\pi \,\,\, \textrm{with $t\in \R$}
\]
and controlled by the probability measure $\pi$.
If we choose a causal MMA process as the model for the volatility of a logarithmic asset price, its dependence structure can be modeled in a versatile way by choosing the distribution $\pi$. Then, the typical decay of the autocovariances of the squared returns, see \cite{CT04}, can be more easily reproduced.

Let the logarithmic asset price $(J_t)_{t \in \R^+}$ be
\begin{equation}
\label{mmasv}
J_t= \int_0^t\sqrt{X_s} dW_s,\,\,\, J_0=0,
\end{equation}
where $(W_t)_{t \in \R^+}$ is a standard Brownian motion and $(X_t)_{t\in \R^{+}}$ is an adapted, stationary and square-integrable causal MMA process with values 
in $\R^{+}$ being independent of $W$. We call (\ref{mmasv}) an MMA SV model.
In the literature, stochastic volatility models where $X$ is given by a sum of independent non-Gaussian OU type process are given in \cite{BNS02,BNS01} which have been later
extended to supOU processes in \cite{BNS11,BNS13}. The latter is an example of an MMA SV model whose dependence structure is going to be analysed in this section. Other financial market models using supOU processes as building blocks and allowing for short and long range dependence can be found in \cite{HL05,K14,L12}.

We show the $\theta$-weak dependence of the return process, over equidistant time intervals $[(t-1)\Delta,t \Delta]$
\begin{equation}
\label{ret}
Y_t=J_{t\Delta}-J_{(t-1)\Delta}=\int_{(t-1)\Delta}^{t\Delta} \sqrt{X_s} dW_s,
\end{equation}
where for convenience of notation we consider $t \in \R$, and the asymptotic normality of its related sample moments.
To this aim, the moments of the return process by using the It\^o isometry, as in \cite{PS09}, turn out to be determined as a function of the moments of the integrated process
\begin{equation}
\label{integrated}
V_t = \int_{(t-1) \Delta}^{t \Delta} X_s\, ds,  
\end{equation}
for $t \in \R$ and $\Delta$ a positive constant. Note that, $(V_t)_{t \in \R}$ corresponds to the integrated volatility process computed over a time interval $[(t-1)\Delta,t \Delta]$.
It is immediate from the definitions (\ref{ret}) and (\ref{integrated}) that the strict stationarity of the process $(X_s)_{s \in \R}$ and its square-integrability imply the stationarity and the square-integrability of the processes $(Y_t)_{t\in \R}$, $(Y_t^2)_{t \in \R}$ and $(V_t)_{t \in \R}$.
Note that, under the square integrability assumption, the moments of the return process can be determined up to the 4th order.

In general, we consider all processes adapted with respect to the filtration $(\mathcal{A}_t)_{t \in \R}$ generated by the set of random variables $\{\Lambda(B): B \in \mathcal{B}(S \times (-\infty,t])\}$ and the increments of the Brownian motion $\{(W_u-W_s)_{s\leq u \leq t} \}$ for all $t \in \R$. 

In order to ensure that the stochastic integrals involving an MMA process as integrand are well defined we assume throughout that
\begin{equation*}
\textrm{(H)}: \left\{\begin{array}{l}
\textrm{The L\'evy basis $\Lambda$ has generating quadruple $(\gamma,0, \nu,\pi)$ such that}\\
\int_{\R} |x| \nu(dx) < \infty,\,\,\ \gamma-\int_{|x|\leq 1} x \nu(dx) \geq 0\,\,\,\textrm{and}\,\,\,\nu(\R^{-})=0;\\
\textrm{the kernel function $f$ is $\mathcal{B}(S \times \R^+)$-measurable,  non-negative}\\
\textrm{and satisfies the assumptions of Corollary \ref{due};}\\
X_s=\int_{-\infty}^s f(A,s-u) \, \Lambda(dA, du) \, \textrm{is adapted and  c\`adl\`ag}. 
\end{array}\right.
\end{equation*}

Sufficient conditions for an MMA process to have c\`adl\`ag sample paths can be found in \cite[Theorem 3.1]{MS13} and the references therein.

We now show the weak dependence properties of the return process.

\begin{Proposition}
\label{theta_return}
Let W be a standard Brownian motion independent of the L\'evy basis $\Lambda$ and assume Assumptions $(H)$ are satisfied. 
Then, the return process defined in (\ref{ret}) is $\theta$-weakly dependent with coefficients
\begin{equation}
\label{coefficients_return}
\theta_Y(r)= \sqrt{\Delta \,\theta_X((r-1)\Delta)},
\end{equation}
where $(\theta_X(r))_{r \in \R^+}$ are the coefficients (\ref{coefficients2_theta}), for all $r \geq 1$.
\end{Proposition}

\begin{proof}
The assumptions (H) imply that the resulting MMA process is non-negative and
the process $Y_t$ is well defined and square-integrable by Corollary \ref{moment2}.
Let $Y_t^{(m)}$ be defined for $m \geq 0$
\[
Y_t^{(m)}=\int_{(t-1) \Delta}^{t \Delta} \sqrt{X_s^{(m)} } \, dW_s
\]
where $X_s^{(m)}$ is defined in (\ref{truncated_theta}). Then,
\begin{equation}
\label{res_truncated}
\E[|Y_t-Y_t^{(m)}|]=\E\Big[\Big | \int_{(t-1) \Delta}^{t \Delta} \sqrt{X_s} - \sqrt{X_s^{(m)}} \, dW_s  \Big |\Big] 
\end{equation}
\[
\leq \E\Big[  \int_{(t-1) \Delta}^{t \Delta} ( \sqrt{X_s} - \sqrt{X_s^{(m)}})^2 \, ds\Big]^{\frac{1}{2}}
\]
by using the inequality $\sqrt{a+b}-\sqrt{a} \leq \sqrt{b} \,\, \textrm{for all $a,b \in \R^+$}$
\[
\leq \E\Big[  \int_{(t-1) \Delta}^{t \Delta} |X_s-X_s^{(m)}| \, ds \Big]^{\frac{1}{2}}
\]
\[
\leq \E\Big[  \int_{(t-1) \Delta}^{t \Delta} \Big |  \int_S \int_{-\infty}^{s-m} f(A,s-u) \,\Lambda(dA,du) \, \Big|\Big]^{\frac{1}{2}}\leq \sqrt{\Delta \theta_X(m)}.
\]
The last inequality follows by (\ref{trunc_theta}).

Let $F$ and $G$ belong respectively to the class of bounded functions $\mathcal{H}^*$ and 
$\mathcal{H}$, $(u,v) \in \mathbb{N}^* \times \mathbb{N}^*$, $r \in \R^{+}$, $(i_1,\ldots,i_u) \in \R^u$ and $(j_1,\ldots,j_v) \in \R^v$ with $i_1\leq\ldots\leq i_u \leq i_u+r\leq j_1\leq \ldots\leq j_v$, $Y_i^*=(Y_{i_1},\ldots,Y_{i_u})$ and
$Y_j^{*(m)}=(Y_{j_1}^{(m)},\ldots,Y_{j_v}^{(m)})$ where for all $m\geq 0$
\[
Y_{i_u}= \int_{(i_u-1) \Delta}^{i_u \Delta}\sqrt{X_s} \,dW_s \,\,\,\textrm{and}\,\,\, Y_{j_1}^{(m)}= \int_{(j_1-1) \Delta}^{j_1 \Delta} \sqrt{X_s^{(m)}}\, dW_s.
\]
Then, if $(j_1-1)\Delta-m-i_u \Delta \geq 0$ that can also be expresses as $j_1-i_u \geq \frac{m}{\Delta}+1$, $I_u=S \times (-\infty,i_u \Delta]$ and $J_1=S \times [(j_1-1)\Delta-m,j_1 \Delta]$ are disjoint sets or have as intersection a set of measure zero when $i_u \Delta=(j_1-1)\Delta-m$, i.e. $\pi\times\lambda(S\times \{i_u \Delta\})=0$. Then, by the definition of a L\'evy basis, the integrands of $Y_{i_u}$ and $Y_{j_1}^{(m)}$ are independent and so are $Y_{i_u}$ and $Y_{j_1}^{(m)}$ because of the independence of $W$ and the L\'evy basis.
Then, the sequences $(Y_i)_{i \leq i_u}$ and $(Y_j^{(m)})_{j\geq j_1}$ are independent and so are $F(Y_i^{*})$ and $G(Y_j^{*(m)})$.
Therefore, let $m=(r-1) \Delta$
\[
|Cov(F(Y_i^*),G(Y_j^*))|\leq |Cov(F(Y_i^*), G(Y_j^{*})-G(Y_j^{*(m)}))|+ |Cov(F(Y_i^*),G(Y_j^{*(m)}))|
\]
\[
\leq 2 \,\E |G(Y_j^*)-G(Y_j^{*(m)})| 
\]
the last relation comes from $\|F\|_{\infty}\leq 1$
\[
\leq 2 \, Lip(G) \sum_{k=1}^v \E|Y_{j_k}-Y_{j_k}^{(m)}| 
\]
using (\ref{res_truncated}) 
\[
=2 v Lip(G) \sqrt{\Delta \,\theta_X((r-1)\Delta)},
\]
which converges to zero as $r$ goes to infinity by applying the dominated convergence theorem.
\end{proof}

Let us consider a supOU process $X$ defined as in (\ref{supou}) such that the underlying L\'evy process $L$ is a subordinator. It can be shown that the process is adapted and c\`adl\`ag under the assumptions $(ii)$ and $(iii)$ of \cite[Theorem 3.12]{BNS11}. Then, Assumptions $\textrm{(H)}$ are satisfied and we can define a supOU SV model and the resulting return process 
\begin{equation}
\label{supousv}
Y_t=\int_{(t-1) \Delta}^{t \Delta} \sqrt{\int_{\R^-} \int_{-\infty}^s e^{A(s-u)} \,\Lambda(dA,du)} dW_s.
\end{equation}
By applying Proposition \ref{theta_return}, $Y$ is $\theta$-weakly dependent with coefficients
\begin{equation}
\label{theta_ret}
\theta_Y(r)=\sqrt{ -\Delta \mu \int_{\R^{-}} \frac{\mathrm{e}^{A\Delta(r-1)}}{A} \, \pi(dA)},
\end{equation}
where $\mu$ is the mean of the underlying L\'evy process.
\\
\\
We consider a sample of $N$ observations of $Y$ and we define the following sample moments for the return process. 

The sample mean
\begin{equation}
\label{mean_ret}
 \frac{1}{N} \sum_{i=1}^N Y_{1+i}, 
\end{equation}
the sample autocovariance function for $k \in \N$
\begin{equation}
\label{timeseries2_ret}
\frac{1}{N}\sum_{j=1}^{N} Y_{1+j}Y_{1+j+k},
\end{equation}
and the fourth order (non-centered) sample moments for $k \in \N$
\begin{equation}
\label{timeseries3_ret}
\frac{1}{N}\sum_{j=1}^{N} Y_{1+j}^2Y_{1+j+k}^2.
\end{equation} 
When the asymptotic properties of the sample autocovariance functions are investigated,
we focus on the processes 
\begin{equation}
\label{cova_ret}
W_{j,k}=Y_{1+j}Y_{1+j+k}-T(k),
\end{equation} 
where we denote by $T(k)$ the covariances of order $k$ defined by 
\[
 T(k)=\E[Y_1Y_{1+k}], 
\]
whereas in the case of the fourth order sample moments on
\begin{equation}
\label{4_ret}
\tilde{W}_{j,k}= Y_{1+j}^2Y_{1+j+k}^2- D^*(k)-\E[V_1]^2,
\end{equation} 
where we denote by $D^*(k)$ the covariances of order $k$ defined by 
\[
 D^*(k)= Cov(V_1,V_{1+k})
\]
where $V$ is the integrated process defined in (\ref{integrated}).

Analogous to Proposition \ref{lem}, we show that the $\theta$-weak dependence of the return process is inherited by the process $Z_t=(Y_{t},Y_{t+1},\ldots,Y_{t+k})$ for all $k \in \N$.

\begin{Lemma}
\label{lem2}
Let $\Lambda$  be a L\'evy basis and $X$ an MMA process satisfying Assumptions $\textrm{(H)}$ and W be a standard Brownian motion independent of $\Lambda$.
We consider the process
\[
Z_t= \left(
 \begin{array}{l}
Y_t\\
\vdots\\
Y_{t+k}
\end{array}
\right)=  \left(
  \begin{array}{l}
\int_{(t-1)\Delta}^{(t)\Delta} \sqrt{X_s} \, dW_s\\
\,\,\,\,\,\,\,\,\vdots\\
\int_{(t+k-1)\Delta}^{(t+k)\Delta} \sqrt{X_s} \, dW_s
\end{array}
\right) = \int_{(t-1)\Delta}^{(t+k)\Delta} G_s \, dW_s
\]
where $Y_t$ is the return process defined in (\ref{ret}) and $G_s$ is an $\R^{k+1}\times \R$ valued process defined as
$G_s =\left(\begin{array}{l}
  \sqrt{X_s} \mathbb{1}_{((t-1)\Delta, t\Delta]}(s)\\
 \sqrt{X_s} \mathbb{1}_{(t\Delta, (t+1)\Delta]}(s)\\
 \,\,\,\,\,\,\,\,\,\,\vdots\,\,\,\,\,\,\,\,\,\,  \\
 \sqrt{X_s}\mathbb{1}_{((t+k-1)\Delta, (t+k)\Delta]}(s)
\end{array}\right)$. 
Then, $Z$ is $\theta$-weakly dependent with coefficients
\[
\theta_Z(r)= \mathcal{D}^{*} \sqrt{\Delta \theta_X((r-k-1) \Delta)}
\]
for $r \geq k+1$ being $\mathcal{D}^{*}= (k+1)$ and $\theta_X$ given in (\ref{coefficients2_theta}).
\end{Lemma}

\begin{proof}
For $m \geq 0$, let $Z_t^{(m)}$ be 
\[
Z_t^{(m)}=\int_{(t-1) \Delta}^{(t+k) \Delta} \sqrt{G_s^{(m)}} \, dW_s
\]
where $G_s^{(m)}=
\left(\begin{array}{l}
\sqrt{X_s^{(m)}} \mathbb{1}_{((t-1)\Delta, t\Delta]}(s)\\
\sqrt{X_s^{(m)}} \mathbb{1}_{(t\Delta, (t+1)\Delta]}(s)\\
 \,\,\,\,\,\,\,\,\,\,\vdots\,\,\,\,\,\,\,\,\,\,\\
\sqrt{X_s^{(m)}}\mathbb{1}_{((t+k-1)\Delta, (t+k)\Delta]}(s)
\end{array}\right)
$.
Then,
\begin{equation}
\label{res_truncated_multi}
\E[\|Z_t-Z_t^{(m)}\|]=\E\Big[\Big \| \int_{(t-1) \Delta}^{(t+k) \Delta} G_s-G_s^{(m)} \, dW_s \Big \| \Big] 
\end{equation}
\[
\leq \E\Big[ \int_{(t-1) \Delta}^{(t+k) \Delta} tr((G_s-G_s^{(m)})(G_s-G_s^{(m)})^{\prime}) \, ds \Big]^{\frac{1}{2}}
\]
by means of the triangular inequality
\[
\leq \E\Big[ \int_{(t-1) \Delta}^{t \Delta} (\sqrt{X_s}-\sqrt{X_s^{(m)}})^2 \, ds \Big]^{\frac{1}{2}}+\ldots+\E\Big[ \int_{(t+k-1) \Delta}^{(t+k) \Delta} (\sqrt{X_s}-\sqrt{X_s^{(m)}})^2 \, ds \Big]^{\frac{1}{2}}
\]
\[
\leq (k+1) \sqrt{\Delta \theta_X(m)}.
\]

Proceeding as in Proposition \ref{lem} and in Proposition \ref{theta_return} the claim follows.

\end{proof}

It can also be shown that the process $Z_t$ is mixing, and thus ergodic, proceeding as in the proof of \cite[Theorem 4.2]{FS13}.

The following asymptotic result holds for (\ref{mean_ret}).
\begin{Theorem}
\label{clt_mean_ret}
We assume that Assumptions $\textrm{(H)}$ hold and that $\int_{|x| >1} |x|^{1+\delta} \nu(dx) < \infty$, for some $\delta>0 $, and $f$ belongs to $L^{1+\delta}(S \times \R^+,\pi \otimes\lambda)\cap L^{1}(S \times \R^+, \pi \otimes \lambda)$.
If $(Y_i)_{i \in \R}$ as defined in (\ref{ret}) is a $\theta$-weakly dependent process such that the volatility process $X$ admits coefficients $\theta_X(r)= O(r^{-\alpha})$ with $\alpha>2\Big(1+\frac{1}{\delta}\Big)$, then $ \sigma^2_Y = Var(Y_1)$ is non-negative and as $N \to \infty$
\begin{equation}
\label{thesis_ret}
\frac{1}{\sqrt{N}}\sum_{i=1}^N Y_{1+i} \stackrel{d}{\rightarrow} \mathcal{N}(0,\sigma^2_Y).
\end{equation}
\end{Theorem}

\begin{proof}
Corollary \ref{moment2} applies and the return process is ergodic because of its mixing properties shown in \cite[Theorem 4.2]{FS13}.
\cite[Theorem 1]{DR00} can be applied, analogously as in Theorem \ref{clt_mean2}, assuring the result (\ref{thesis_ret}) where the asymptotic variance of the
sample mean is given by the absolute summable series $\sum_{l\in Z} Cov(Y_1,Y_{1+l})=Var(Y_1)$.
\end{proof}

Applying Proposition \ref{lem2} and Proposition \ref{her3}, the following can be easily shown.

\begin{Proposition}
\label{theta2_ret}
We assume that Assumptions $\textrm{(H)}$ hold and that {\small $\int_{\|x\| >1} \|x\|^{2+\delta} \nu(dx) < \infty$}, for some $\delta>0 $, and $f$ belongs to $L^{2+\delta}(S \times \R^+,\pi \otimes \lambda) \cap L^{1}(S \times \R^+, \pi \otimes \lambda)$. If $(Y_i)_{i \in \Z}$ as defined in (\ref{ret}) is $\theta$-weakly dependent process with coefficients $\theta_Y$ as in (\ref{coefficients_return}), then, for all $k >0$  $\mathcal{Z}=(W_{j,0}, W_{j,1},\ldots, W_{j,k})_{j \in \Z}$ is $\theta$-weakly dependent with coefficients
\[
\theta_{\mathcal{Z}}(r)= \mathcal{C} (\mathcal{D}^* \sqrt{\Delta \theta_X((r-k-1)\Delta)})^{\frac{\delta}{1+\delta}},
\]
Moreover, if we assume that $\int_{\|x\| >1} \|x\|^{4+\delta} \nu(dx)< \infty$, for some $\delta>0 $, and $f$ belongs to $L^{4+\delta}(S \times \R,\pi \otimes \lambda) \cap L^{1}(S \times \R, \pi\otimes \lambda)$, then the process $\tilde{\mathcal{Z}}= (\tilde{W}_{j,0}, \tilde{W}_{j,1},\ldots, \tilde{W}_{j,k})_{j \in \Z}$ is $\theta$-weakly dependent with coefficients
\[
\theta_{\tilde{\mathcal{Z}}}(r)= \mathcal{C} (\mathcal{D}^* \sqrt{\Delta \theta_X((r-k-1)\Delta)})^{\frac{\delta}{3+\delta}}.
\]
The constant $\mathcal{C}$ is independent of $r$ and $\mathcal{D}^*=k+1$ in the above formulas.
\end{Proposition}

\begin{Corollary}
\label{clt_multi_acf_ret}
We assume that Assumptions $\textrm{(H)}$ hold and that {\small$\int_{\|x\| >1} \|x\|^{2+\delta} \nu(dx) < \infty$}, for some $\delta>0 $, and $f$ belongs to 
$L^{2+\delta}(S \times \R^+, \pi \otimes \lambda)\cap L^{1}(S \times \R^+, \pi \otimes \lambda)$.
If $(\mathcal{Z}_i)_{i \in \Z}$ as defined in Proposition \ref{theta2_ret} is $\theta$-weakly dependent such that the volatility process $X$ admits coefficients $\theta_X(r)= O(r^{-\alpha})$ 
with $\alpha>(1+\frac{1}{\delta})(\frac{2+2\delta}{\delta})$,  then for each $p,q \in \{0,\ldots,k\}$ with $k \in \N$,
\[
\sum_{l \in \N} Cov(W_{0,p}, W_{l,q}) 
\]
are finite and as $N \to \infty$
\[
\frac{1}{\sqrt{N}}\sum_{j=1}^{N} \mathcal{Z}_j \stackrel{d}{\rightarrow} \mathcal{N}_{k+1}(0,\Psi)
\]
where $\Psi$ is equal to
\begin{equation*}
\footnotesize
\left[\begin{array}{cccc}
\sum_{l \in \Z} Cov(Y_1^2,Y_{l+1}^2)& \sum_{l\in \Z} Cov(Y_1^2,Y_{l+1}Y_{l+2})& \ldots & \sum_{l \in \Z} Cov(Y_1^2,Y_{l+1}Y_{l+k+1})\\
\ldots &\sum_{l \in \Z} Cov(Y_{1}Y_{2},Y_{l+1} Y_{l+2}) & \ldots & \sum_{l \in \Z} Cov(Y_{1}Y_{2},Y_{l+1}Y_{l+k+1})\\
\ldots& \ldots &\ldots & \sum_{l \in \Z} Cov(Y_{1}Y_{k+1},Y_{l+1} Y_{l+k+1})
\end{array}\right].
\end{equation*}
and positive semidefinite.
\end{Corollary}

\begin{proof}
Since Corollary \ref{moment2} holds, $Z$ is a $\theta$-weakly dependent process with coefficients given in Proposition \ref{theta2_ret}$, \theta_Z(r)= \mathcal{C}(\mathcal{D}^*\sqrt{\Delta\theta_X((r-k-1)\Delta)})^{\frac{\delta}{1+\delta}}$. 
For all $a \in \R^{k+1}$, $a^{\prime}Z$ is a $\theta$-weakly dependent process, because a linear function is Lipschitz, and ergodic having the same $\theta$-coefficients as the process $\mathcal{Z}$.
Under the assumptions of the Corollary, \cite[Theorem 1]{DR00} is then applied  and
\[
\frac{1}{\sqrt{N}} \sum_{j=1}^N a^{\prime}\mathcal{Z}_j \xrightarrow{d} \mathcal{N}(0, a^{\prime} \Psi a) 
\]
as $N \to \infty$.
Applying the Cramer-Wold device, the asymptotic normality of the vector $\mathcal{Z}$ is shown.
\end{proof}

\begin{Corollary}
\label{clt_multi_4_ret}
We assume that Assumptions $(H)$ hold and that {\small$\int_{\|x\| >1} \|x\|^{4+\delta} \nu(dx) < \infty$}, for some $\delta>0 $, and $f$ belongs to $L^{4+\delta}(S \times \R^+, \pi \otimes \lambda)\cap L^{1}(S \times \R^+, \pi \otimes \lambda)$.
If $(\tilde{\mathcal{Z}}_i)_{i \in \Z}$ as defined in Proposition \ref{theta2_ret} is $\theta$-weakly dependent such that the volatility process $X$ admits coefficients $\theta_X(r)= O(r^{-\alpha})$ 
with $\alpha>(1+\frac{1}{\delta})(\frac{6+2\delta}{\delta})$, then  for each $p,q \in \{0,\ldots,k\}$ with $k \in \N$,
\[
\sum_{l \in \N} Cov(\tilde{W}_{0,p}, \tilde{W}_{l,q}) 
\]
are finite and as $N \to \infty$
\[
\frac{1}{\sqrt{N}}\sum_{j=1}^{N} \tilde{\mathcal{Z}}_j \stackrel{d}{\rightarrow} \mathcal{N}_{k+1}(0,\mathcal{M})
\]
where $\mathcal{M}$ is equal to 
\begin{equation*}
\footnotesize
\left[\begin{array}{cccc}
\sum_{l \in \Z} Cov(Y_1^4,Y_{l+1}^4)& \sum_{l\in \Z} Cov(Y_1^4,Y_{l+1}^2Y_{l+2}^2)& \ldots & \sum_{l \in \Z} Cov(Y_1^4,Y_{l+1}^2Y_{l+k+1}^2)\\
\ldots &\sum_{l \in \Z} Cov(Y_{1}^2Y_{2}^2,Y_{l+1}^2 Y_{l+2}^2) & \ldots & \sum_{l \in \Z} Cov(Y_{1}^2Y_{2}^2,Y_{l+1}^2Y_{l+k+1}^2)\\
\ldots& \ldots &\ldots & \sum_{l \in \Z} Cov(Y_{1}^2Y_{k+1}^2,Y_{l+1}^2 Y_{l+k+1}^2)
\end{array}\right].
\end{equation*}
is positive semidefinite.
\end{Corollary}

\begin{proof}
The proof follows as in Corollary \ref{clt_multi_acf_ret}, using the $\theta$ coefficients of the process $\tilde{\mathcal{Z}}$ as determined in Proposition \ref{theta2_ret}.
\end{proof}

\begin{Remark}
\label{cum}
In view of Section \ref{sec5}, let us give explicit formulas of the third and fourth order cumulant of an integrated process $V$ and of the covariances $Cov(\tilde{W}_{0,p},\tilde{W}_{l,q})$ for $p,q \in \{0,\ldots,k\}$ and $k \in \N$ under the assumptions of Corollary \ref{clt_multi_4_ret}.
Let us consider an integrated process as defined in (\ref{integrated}) with mean $\E[V_1]=C^*$.
For all $(i,j,k,l) \in \R^4$, we call $K(i,j,k)$ and $K(i,j,k,l)$ the centered cumulant or order three and four which are respectively equal to
\begin{equation}
\label{cumulant3_int}
K(i,j,k)=\int_{i \Delta}^{(i+1)\Delta} \int_{j \Delta}^{(j+1) \Delta} \int_{k \Delta}^{(k+1) \Delta} \kappa(s,t,u) \, ds\,dt \, du,
\end{equation}
with $\kappa(s,t,u)$ given in (\ref{cumulant3}), and
\begin{equation}
\label{cumulant_int}
K(i,j,k,l)=\int_{i \Delta}^{(i+1)\Delta} \int_{j \Delta}^{(j+1) \Delta} \int_{k \Delta}^{(k+1) \Delta} \int_{l \Delta}^{(l+1)\Delta} \kappa(s,t,u,z) \, ds\,dt \, du \, dz,
\end{equation}
where $\kappa(s,t,u,z)$ is defined in (\ref{cumulant}).
Moreover,
\begin{equation}
D^*(l)=Cov(V_1,V_{1+l})=\int_0^{\Delta} \int_{l \Delta}^{(l+1) \Delta} D(u-s) ds du,
\end{equation}
where $D(\cdot)$ is the covariance function, as in (\ref{passo}), of the centered MMA process underlying the definition of the integrated process.Hence, by means of the It\^o formula, the independence of the process $(X_t)_{t \in \R}$ with $(W_t)_{t \in \R}$ and using arguments similar to the formula $(40)$ in \cite{PS09} 
\[
Cov(\tilde{W}_{0,p},\tilde{W}_{l,q})
\]
\[
= K(0,p,l,l+q)+ C^*(K(0,p,l)+K(0,p,l+q)+K(p,l,l+q)+K(0,l,l+q))
\]
\[
 +C^{*2}(D^*(l-p)+D^*(l+q-p)+D^*(l)+D^*(l+q))+D^*(l)D^*(l+q-p)
\]
\[
+D^*(l+q)D^*(l-p)+A(0,p,l,l+q),
\]
where $A(i,j,k,l)$ is defined in Table \ref{tab4mom} for $(i,j,k,l) \in \R^4$.

\begin{table}
\footnotesize
\centering
\begin{tabular}{|c|c|}
\hline
$(i,j,k,l)$ & $A(i,j,k,l)$\\
\hline
$\{l \neq i \neq j \neq k \}$ & $0$\\
\hline
 $\{i\neq j\} \wedge \{j=k=l\}$ & $ 12 \int_{i\Delta}^{(i+1)\Delta} \int_{j \Delta}^{(j+1)\Delta} \int_{j \Delta}^{(j+1)\Delta} \int_{j \Delta}^s \E[X_t X_z X_s X_u] du ds dz dt $\\
\hline
$\{i \neq j \neq k\}\wedge \{ k=l \}$ & $ 4 \int_{i\Delta}^{(i+1)\Delta} \int_{j \Delta}^{(j+1)\Delta} \int_{k \Delta}^{(k+1)\Delta} \int_{k \Delta}^s \E[X_t X_z X_s X_u] du ds dz dt $\\
\hline
$\{i=j\} \wedge \{ k=l\} \wedge \{i\neq k \}$ & $ 4 \int_{i\Delta}^{(i+1)\Delta} \int_{i \Delta}^{(i+1)\Delta} \int_{k \Delta}^{(k+1)\Delta} \int_{k\Delta}^s \E[X_t X_z X_s X_u] du ds dz dt$\\
&$+4 \int_{i\Delta}^{(i+1)\Delta} \int_{i \Delta}^{t} \int_{k \Delta}^{(k+1)\Delta} \int_{k \Delta}^{(k+1)\Delta} \E[X_t X_z X_s X_u] du ds dz dt$\\
 & $ + 16 \int_{i \Delta}^{(i+1)\Delta} \int_{k \Delta}^{(k+1)\Delta} \int_{i \Delta}^t \int_{k \Delta}^{s} \E[X_t X_s X_z X_u] du ds dz dt$ \\
\hline 
$\{i=j=k=l\}$ & $ \E\Big[24 \Big(\int_{i\Delta}^{(i+1) \Delta}  X_t  \, dt \Big)^2 \int_{i\Delta}^{(i+1)\Delta} X_s \int_{i \Delta}^s X_u du ds \Big]$\\
 & $+ \E\Big[ 96 \int_{i \Delta}^{(i+1) \Delta} \int_{i \Delta}^{s} X_u \int_{i \Delta}^u X_z \,dz du ds \Big] $\\
\hline
\end{tabular}
\caption{Explicit closed formula for the summand $A(i,j,k,l)$ for $(i,j,k,l) \in \Z^4$.}
\label{tab4mom}
\end{table}
\end{Remark}

\section{Generalized method of moments for the supOU process and supOU SV model}
\label{sec5}

In this section, we apply the developed asymptotic theory to determine the asymptotic normality of GMM estimators
of the supOU process and of the supOU SV model defined in \cite{STW15}.

Let $X$ and $Y$ be a supOU process and a return process as respectively defined in (\ref{supou}) and  (\ref{supousv}).
We assume, $$\int_{|x| > 1} x^2 \nu(dx) < \infty,$$ then, as shown in \cite[Theorem 2.3 and Theorem 2.8]{STW15}, the supOU process $X$ and the return process $Y$ have known moments given by
\begin{align*}
\E[X_0]=& - \mu \int_{\R^{-}} \frac{1}{A} \, \pi(dA), \,\,\, \mathrm{Var}[X_0]=-\sigma^2\int_{\R^{-}} \frac{1}{2A} \, \pi(dA), \\ &\mathrm{Cov}[X_0,X_{k\Delta}]=-\sigma^2 \int_{\R^{-}} \frac{\mathrm{e}^{Ak\Delta}}{2A} \, \pi(dA),
\end{align*}
and 
\begin{align*}
&\E[Y_1]=0,\,\,\, Var[Y_1]=-\Delta \mu \int_{\R^-} \frac{1}{A} \, \pi(dA),\,\,\,Cov(Y_1,Y_{1+k})=0, \\
&\E[Y_1^2]=-\Delta \mu \int_{\R^-} \frac{1}{A}\, \pi(dA), \\
 &Var(Y_1^2)=-3 \sigma^2 \int_{\R^-} \frac{1}{A^2}\Big(\frac{e^{A\Delta}}{A}-\frac{1}{A}-\Delta\Big)\,\pi(dA) +2 \Big( - \Delta \mu \int_{\R^-} \frac{1}{A} \, \pi(dA) \Big)^2, \\
&Cov(Y_1^2,Y_{1+k}^2)=-\sigma^2 \int_{\R^-}\frac{1}{2A^3} (s_{k+1}-2s_k+s_{k-1}) \, \pi(dA),
\end{align*}
where $\mu= \E[L_1]$ and $\sigma^2= Var[L_1]$ are the mean and variance of the underlying L\'evy process and $s_k:= e^{A \Delta k}$.

\begin{Assumption}
\label{assumption0}
Let us assume that the mean reversion parameter $A$ is Gamma distributed. That is, we assume that $\pi$ is the distribution of $B\xi$ where $B \in \R^{-}$ and $\xi$ is $\Gamma(\alpha_{\pi},1)$ distributed with $\alpha_{\pi}>2$. 
\end{Assumption}

We emphasize that setting the second parameter of the Gamma distribution equal to one does not restrict the model since this is equivalent to varying $B$. Under Assumption \ref{assumption0}, we observe the decay of the autocovariances of the supOU process as given in Remark \ref{decay}, notice that in this set-up $\alpha=\alpha_{\pi}-1$.

The moments of the supOU process $X$ and of the return process $Y$, under Assumption \ref{assumption0}, have been computed in \cite[Section 2.2]{STW15} 
\begin{align}
\label{moments1}
\begin{split}
\E[X_0]=& - \frac{\mu}{B (\alpha_{\pi}-1)}, \,\,\, \mathrm{Var}[X_0]=-\frac{\sigma^2}{2B (\alpha_{\pi}-1)}, \\ &\mathrm{Cov}[X_0,X_{k\Delta}]=-\frac{\sigma^2(1-B  k\Delta)^{1-\alpha_{\pi}}}{2 B  (\alpha_{\pi}-1)},
\end{split}
\end{align}
\begin{align}
\label{moments2}
\begin{split}
&\E[Y_1]=0 \,\,\, Var(Y_1)=-\frac{\Delta \mu}{B (\alpha_{\pi}-1)} , \,\,\, Cov(Y_1,Y_{1+k})=0, \\
&\E[Y_1^2]=-\frac{\Delta \mu}{B (\alpha_{\pi}-1)},\\
& Var(Y_1^2)=-3 \sigma^2 \frac{(1-B \Delta)^{3-\alpha_{\pi}}-1-\Delta B(\alpha_{\pi}-3)}{B^3(\alpha_{\pi}-1)(\alpha_{\pi}-2)(\alpha_{\pi}-3)} +2 \Big(-\frac{\Delta \mu}{B (\alpha_{\pi}-1)} \Big)^2, \\
&Cov(Y_1^2,Y_{1+k}^2)=-\frac{\sigma^2 (f_{k+1}-2f_k+f_{k-1})}{2B^3 (\alpha_{\pi}-1)(\alpha_{\pi}-2)(\alpha_{\pi}-3)},
\end{split}
\end{align}
where $f_k:= (1-B\Delta k)^{3-\alpha_{\pi}}$.
Therefore, the moment structure, up to the 2nd order for the supOU process and up to the 4th order for the return process, depends only on the parameter vector $\theta:=(\mu,\sigma^2,\alpha_{\pi},B)$.

\subsection{SupOU process}
Suppose we observe a sample $\{X_t: t=1\Delta,\ldots,N\Delta \}$ for the supOU process with $\Delta$ a positive constant. We construct the following moment functions, as in \cite{STW15}, by using the auto-covariances up to a lag $m \geq 2$ of $X$.

First, let us define $X_t^{(m)}=(X_{t\Delta}, X_{(t+1)\Delta},\ldots,X_{(t+m)\Delta})$ for $t=1,\ldots,N-m$ and the measurable function
$h: \R^{m+1} \times \Theta \to \R^{m+2}$ as
\begin{equation*}
h(X_t,\theta)=\left(\begin{array}{l}
h_{\E}(X_t^{(m)},\theta)\\
h_{0}(X_t^{(m)},\theta)\\
\,\,\,\,\,\,\,\,\vdots\\
h_k(X_t^{(m)},\theta)\\
\,\,\,\,\,\,\,\,\vdots\\
h_m(X_t^{(m)},\theta)
\end{array}\right)
\end{equation*}
\begin{equation}
\label{mom_func1}
=\left(\begin{array}{l}
X_{t\Delta}+\frac{\mu}{B(\alpha_{\pi}-1)}\\
X_{t\Delta}^2-\Big(\frac{\mu}{B(\alpha_{\pi}-1)}\Big)^2+\frac{\sigma^2}{2B(\alpha_{\pi}-1)}\\
X_{t\Delta}X_{(t+1)\Delta}-\Big(\frac{\mu}{B(\alpha_{\pi}-1)}\Big)^2+\frac{\sigma^2 (1-B\Delta)^{1-\alpha_{\pi}}}{2B(\alpha_{\pi}-1)}\\
\,\,\,\,\,\,\,\,\vdots\\
X_{t\Delta}X_{(t+k)\Delta}-\Big(\frac{\mu}{B(\alpha_{\pi}-1)}\Big)^2+ \frac{\sigma^2 (1-Bk\Delta)^{1-\alpha_{\pi}}}{2B(\alpha_{\pi}-1)}\\
\,\,\,\,\,\,\,\,\vdots\\
X_{t\Delta}X_{(t+m)\Delta}-\Big(\frac{\mu}{B(\alpha_{\pi}-1)}\Big)^2+ \frac{\sigma^2 (1-Bm \Delta)^{1-\alpha_{\pi}}}{2B(\alpha_{\pi}-1)}
\end{array}\right).
\end{equation}

We can now define the sample moment function for the supOU process as

\begin{equation}
g_{N,m}(X,\theta)=\left(\begin{array}{l}
\frac{1}{N-m} \sum_{t=1}^{N-m} h_{\E}(X_t^{(m)},\theta)\\
\frac{1}{N-m} \sum_{t=1}^{N-m} h_{0}(X_t^{(m)},\theta)\\
\,\,\,\,\,\,\,\,\vdots\\
\frac{1}{N-m} \sum_{t=1}^{N-m} h_k(X_t^{(m)},\theta)\\
\,\,\,\,\,\,\,\,\vdots\\
\frac{1}{N-m} \sum_{t=1}^{N-m} h_m(X_t^{(m)},\theta)
\end{array}\right)
\end{equation}
\[
=\left(\begin{array}{l}
\frac{1}{N-m} \sum_{t=1}^{N-m} \Big(X_{t\Delta}+\frac{\mu}{B(\alpha_{\pi}-1)}\Big) \\
\frac{1}{N-m} \sum_{t=1}^{N-m} \Big(X_{t\Delta}^2-\Big(\frac{\mu}{B(\alpha_{\pi}-1)}\Big)^2+\frac{\sigma^2}{2B(\alpha_{\pi}-1)}\Big)\\
\frac{1}{N-m} \sum_{t=1}^{N-m} \Big(X_{t\Delta}X_{(t+1)\Delta}-\Big(\frac{\mu}{B(\alpha_{\pi}-1)}\Big)^2+\frac{\sigma^2 (1-B\Delta)^{1-\alpha_{\pi}}}{2B(\alpha_{\pi}-1)}\Big)\\
\,\,\,\,\,\,\,\,\vdots\\
\frac{1}{N-m} \sum_{t=1}^{N-m} \Big(X_{t\Delta}X_{(t+k)\Delta}-\Big(\frac{\mu}{B(\alpha_{\pi}-1)}\Big)^2+ \frac{\sigma^2 (1-Bk\Delta)^{1-\alpha_{\pi}}}{2B(\alpha_{\pi}-1)}
 \Big)\\
\,\,\,\,\,\,\,\,\vdots\\
\frac{1}{N-m} \sum_{t=1}^{N-m} \Big(X_{t\Delta}X_{(t+m)\Delta}-\Big(\frac{\mu}{B(\alpha_{\pi}-1)}\Big)^2+ \frac{\sigma^2 (1-Bm\Delta)^{1-\alpha_{\pi}}}{2B(\alpha_{\pi}-1)} \Big)
\end{array}\right),
\]
and estimate $\theta_0$ by minimizing the objective function
\begin{equation}
\label{estimator}
\hat{\theta}_{0}^{N,m}= \mathrm{argmin} \, g_{N,m}(X,\theta)^{\prime} A_{N,m} g_{N,m}(X,\theta)
\end{equation}
where $A_{N,m}$ is a positive definite matrix to weight the $m+2$ different moments collected in 
$g_{N,m}(X,\theta)$.

We aim to show the asymptotic normality of the GMM estimator (\ref{estimator}).
Hence, we first need to show that the moment function $h(X_t,\theta_0)$ satisfies a central limit theorem.
\begin{Theorem}
\label{asy_mom1}
Let $\Lambda$ be a real valued L\'evy basis with generating quadruple $(\gamma,\Sigma,\nu,\pi)$ satisfying assumptions (\ref{supou_ass}) such that $\int_{|x|>1} |x|^{4+\delta} \, \nu(dx) < \infty$ for some $\delta >0$ and let Assumption \ref{assumption0} hold with $\alpha_{\pi}-1 > (1+\frac{1}{\delta})(\frac{6+2\delta}{2+\delta})$. Let $X$ be the resulting supOU process, then $h(X_t,\theta_0)$ is a $\theta$-weakly dependent process, 
\begin{equation}
\label{cov_h}
H_{\Sigma}=\sum_{l\in \Z} Cov(h(X_0,\theta_0),h(X_l,\theta_0))
\end{equation}
is finite, positive semidefinite and as $N \to \infty$
\begin{equation}
\label{asy_h}
\sqrt{N}g_{N,m}(X,\theta_0) \xrightarrow{d} \mathcal{N}(0, H_{\Sigma}).
\end{equation}
\end{Theorem}

\begin{proof}
Proposition \ref{moment1} shows that the $4+\delta$-th moments of the supOU process exist. We call $C:=-\frac{\mu}{B(\alpha_{\pi}-1)}$ and, following the notations of Section \ref{sec3}, $D(k)=-\frac{\sigma^2 (1-B\Delta k)^{1-\alpha_{\pi}}}{2B(\alpha_{\pi}-1)}$ for $k=0,\ldots,m$. 
We then consider the vector $Z=(X_{t\Delta},X_{(t+1)\Delta},\ldots,X_{(t+m)\Delta})$ and the function $f:\R^{m+1}\to\R^{m+2}$ such that
\[
f(Z_t)=f\left(
  \begin{array}{l}
\,\,\,\,X_{t\Delta} \\
\,\,\,\,X_{(t+1)\Delta}\\
\,\,\,\,\ldots\\
\,\,\,\,X_{(t+k)\Delta}\\
\,\,\,\,\ldots\\
\,\,\,\,X_{(t+m)\Delta}
\end{array}
\right)=h(X_t,\theta_0)+\left(
  \begin{array}{l}
\,\,\,\,C \\
\,\,\,\,D(0)+C^2\\
\,\,\,\,\ldots\\
\,\,\,\,D(k)+C^2\\
\,\,\,\,\ldots\\
\,\,\,\,\ldots\\
\,\,\,\,D(m)+C^2
\end{array}
\right).
\]
$Z$ is a $\theta$-weakly dependent process with coefficients 
\[
 \theta_Z(r)= \mathcal{D} \theta_X(r-m\Delta)
\]
where $\theta_X$ is given in Formula (\ref{theta_sup_1}). 
The assumptions of Proposition \ref{her3} hold with $p=4+\delta$, $c=1$, $a=2$, thus
$f(Z_t)$ is a $\theta$-weakly dependent process with coefficients 
$\mathcal{C}(\mathcal{D}\theta_X(r-m\Delta))^{\frac{2+\delta}{3+\delta}}$ for
$r \in \N$. Hence, $h(X_t,\theta_0)$ is a $\theta$-weak dependent process with the same coefficients and mean zero. We have that $\theta_h(r)=\mathcal{C} \mathcal{D}^\frac{2+\delta}{3+\delta}\Big(-\frac{\sigma^2 (1-B(2r-2m\Delta))^{1-\alpha_{\pi}}}{2B(\alpha_{\pi}-1)} +\Big(\frac{-2\mu (1-B(r-m\Delta))^{1-\alpha_{\pi}}}{2B(\alpha_{\pi}-1)}\Big)^2\Big)^{\frac{2+\delta}{6+2\delta}}$, where $\alpha_{\pi}$ satisfies the inequality $(\alpha_{\pi}-1)>2(1+\frac{1}{\delta})(\frac{3+\delta}{2+\delta})$ by assumption. Analogously to the proof of Corollary \ref{clt_multi_acf}, by applying \cite[Theorem 1]{DR00} and the Cramer-Wold device, we obtain the distributional result (\ref{asy_h}).
\end{proof}

\begin{Remark}
\label{Sigma}
The coefficients of the matrix $Cov(h(X_0,\theta_0),h(X_l,\theta_0))$ for $l \in \Z$ are 
\begin{equation}
\label{11}
Cov(h_{\E}(X_0^{(m)},\theta_0),h_{\E}(X_l^{(m)},\theta_0)= D(l),
\end{equation}
\begin{equation}
\label{1p}
Cov(h_{\E}(X_0^{(m)},\theta_0),h_{p}(X_l^{(m)},\theta_0))=\kappa(0,l,l+p)+C (D(l)+D(l+p)),
\end{equation}
\begin{equation}
\label{pq}
Cov(h_{p}(X_0^{(m)},\theta_0),h_{q}(X_l^{(m)},\theta_0))=\kappa(0,p,l,l+q) 
\end{equation}
\[+C(\kappa(0,p,l)+\kappa(0,p,l+q)+\kappa(p,l,l+q)+\kappa(0,l,l+q))\]
\[+C^2(D(l-p)+D(l+q-p)+D(l)+D(l+q))+D(l)D(l+q-p)+D(l+q)D(l-p),\]
\noindent
for $p,q \in \{0,\ldots,m\}$, where $C$ and $D(l)$ are defined in the proof of Theorem \ref{asy_mom1} and $\kappa(i,j,k)$ and $\kappa(i,j,k,l)$ are respectively the cumulants of the supOU process of order three and four for $i,j,k,l \in \{0,\ldots,m\}$ defined in (\ref{cumulant3}) and (\ref{cumulant}).

In Table \ref{tab1}, the explicit expressions of the autocovariances and cumulants of the supOU process are reported.

\begin{sidewaystable}
\centering
\begin{tabular}{|c|c|}
\hline
$\{l\geq 0\}$ & $\{l<0\}$\\
$Cov(X_0,X_{l\Delta})=\frac{-\sigma^2(1-B \Delta l)^{1-\alpha_{\pi}}}{2B(\alpha_{\pi}-1)}$ & $Cov(X_0,X_{l\Delta})= \frac{-\sigma^2(1+B \Delta l)^{1-\alpha_{\pi}}}{2B(\alpha_{\pi}-1)}$\\
\hline
$\{l\geq 0\}$ & $\{l<0\}$\\
$\kappa(0,l,l+p)=\frac{-s \sigma^3 (1-B\Delta(2l+p))^{1-\alpha_{\pi}}}{3B(\alpha_{\pi}-1)}$& $\kappa(0,p,l)=\frac{-s \sigma^3(1-B\Delta(p-l))^{1-\alpha_{\pi}}}{3B(\alpha_{\pi}-1)}$\\
\hline
$ \{l\geq p\}$ & $\{ 0<l<p\} \vee \{l\leq 0\} $ \\
$\kappa(p,l,l+q)=\frac{-s \sigma^3(1-B\Delta(2l-2p+q))^{1-\alpha_{\pi}}}{3B(\alpha_{\pi}-1)}$ & $\kappa(p,l,l+q)=\frac{-s \sigma^3(1-B\Delta(p-l+q))^{1-\alpha_{\pi}}}{3B(\alpha_{\pi}-1)}$\\
\hline
$ \{l\geq 0\}$ &  $\{l<0\}$ \\
$\kappa(0,p,l,l+q)=-\frac{(\eta-3)\sigma^4}{4B(\alpha_{\pi}-1)} (1-B\Delta(p+2l+q))^{1-\alpha_{\pi}}$ & $\kappa(0,p,l,l+q)=-\frac{(\eta-3)\sigma^4}{4B(\alpha_{\pi}-1)} (1-B\Delta(p+q-2l))^{1-\alpha_{\pi}}$\\
\hline
\end{tabular}
\caption{Explicit closed formula for the summands appearing in (\ref{11}),(\ref{1p}) and (\ref{pq}), where $p \leq q$ and $\mu=\E[L_1]$, $\sigma^2= Var[L_1]$, $s=\sigma^{-3}\E[(L_1-\mu)^3]$ and $\eta= \sigma^{-4}\E[(L_1-\mu)^4]$.}
\label{tab1}
\end{sidewaystable}

\end{Remark}

\begin{Corollary}
\label{asy_mom2}
Let $\Lambda$ be a real valued L\'evy basis with generating quadruple $(\gamma,0,\nu,\pi)$ such that $\int_{\R} |x| \, \nu(dx) < \infty$ and
$ \int_{|x|>1} |x|^{4+\delta} \, \nu(dx) < \infty$ for some $\delta>0$ and let Assumption \ref{assumption0} hold with $\alpha_{\pi}-1 > (1+\frac{1}{\delta})(\frac{3+\delta}{2+\delta})$. 
Let $(X_t)_{t \in \R}$ be the resulting supOU process, then $h(X_t,\theta_0)$ is a $\theta$-weakly dependent process,
\[
H_{\Sigma}=\sum_{l\in \Z} Cov(h(X_0,\theta_0),h(X_l,\theta_0))
\]
is finite, positive semidefinite and as $N \to \infty$
\[
\sqrt{N}g_{N,m}(X,\theta_0) \xrightarrow{d} \mathcal{N}(0, H_{\Sigma}).
\]
\end{Corollary}
The proof of the Corollary follows the same steps as Theorem \ref{asy_mom1}. 

\begin{Remark}\label{rem65}
Under the assumptions of Corollary \ref{asy_mom2}, a slower decay of the autocovariances of a supOU process is required to obtain asymptotic normality compared to Theorem \ref{asy_mom1}. 
Moreover, if all the moments of the underlying L\'evy process exist then the asymptotic result of Corollary \ref{asy_mom2} holds assuming that $\alpha_{\pi}>2$. The latter assumption results in the slowest decay of the autocovariances of $X$ that can be reached under short memory, see Remark \ref{decay} remembering that in the notations of this section $\alpha=\alpha_{\pi}-1$, whereas the asymptotic result of Theorem \ref{asy_mom1} holds, under these assumptions, for $\alpha_{\pi}>3$.
\end{Remark}

Several assumptions have to be made to show the asymptotic normality of the GMM estimator (\ref{estimator}):
\begin{Assumption}
\label{assumption1}
The parameter space $\Theta$ is compact and large enough to include the true parameter vector $\theta_0$. 
\end{Assumption}
\begin{Assumption}
\label{assumption2}
The matrix $A_{N,m}$ converges in probability to a positive definite matrix of constants $A$.
\end{Assumption}
\begin{Assumption}
\label{assumption_pos1}
The matrix $H_{\Sigma}$ is positive definite.
\end{Assumption}
In our set-up we always need to choose a parameter space such that $\mu \geq 0$, $\sigma^2 > 0$, $\alpha_{\pi}>2$ and $B<0$.
However, Assumption \ref{assumption1} remains reasonable, due to the fact that typically an optimization procedure is used to determine $\hat{\theta}^{N,m}$ and then some parameter bounds are always imposed in practice. 

\begin{Theorem}
\label{GMM_asy}
Let $\Lambda$ be a real valued L\'evy basis with generating quadruple $(\gamma,\Sigma,\nu,\pi)$ and $X$ a supOU process satisfying assumptions (\ref{supou_ass}) such that $\int_{|x|>1} |x|^{4+\delta} \, \nu(dx) < \infty$ for some $\delta >0$. Let Assumption \ref{assumption0} hold with $\alpha_{\pi}-1 >(1+\frac{1}{\delta})(\frac{6+2\delta}{2+\delta})$ or $\alpha_{\pi}-1 > (1+\frac{1}{\delta})(\frac{3+\delta}{2+\delta})$ if in addition $\int_{\R} |x| \, \nu(dx) < \infty$. Moreover, if Assumption \ref{assumption1}, \ref{assumption2} and \ref{assumption_pos1} hold, then as $N$ goes to infinity
\[
\sqrt{N}(\hat{\theta}_{0}^{N,m}-\theta_0) \xrightarrow{d} \mathcal{N}(0, M H_{\Sigma} M^{\prime})
\]
where
\begin{align*}
M= (G_0^{\prime} A G_0)^{-1}G_0^{\prime} A,&\,\,\,\, G_0=\E[\frac{\partial h(X_t,\theta)}{\partial \theta^{\prime}}]_{\theta=\theta_0},\\
\textrm{and} \,\,\, H_{\Sigma}=\sum_{l\in \Z} Cov&(h(X_0,\theta_0),h(X_l,\theta_0)).
\end{align*}

\end{Theorem}

\begin{proof}
We follow the steps of the proof of \cite[Theorem 1.2]{M99}. This involves checking Assumption 1.1-1.9 in \cite{M99}. 
Note that Assumption 1.1-1.3 in \cite{M99} are sufficient to give the consistency of the estimator, as obtained in \cite[Theorem 3.2]{STW15}. We show them for completeness. 
Assumption 1.1(i) is fulfilled by the function $h(X_t,\theta)$ by construction. Assumption 1.1.(ii) is satisfied since the true parameter vector $\theta_0$ is identifiable as shown in \cite[Proposition 3.3]{STW15}.
Asssumption 1.2 in \cite{M99} requires that $\sup_{\theta \in \Theta} |g^{(i)}_{N,m}(X,\theta)-\E[h^{(i)}(X_t,\theta)]| \xrightarrow{p} 0$ as $N \to \infty$, for all $i=1,\ldots,m+2$, i.e. each vector component of $g_{N,m}(X,\theta)-\E[h(X_t,\theta)]$ converges uniformly in probability to zero for each $\theta \in \Theta$. Assumption 1.4-1.6 in \cite{M99} represent three sufficient conditions that if fulfilled imply as consequence Assumption 1.2 in \cite{M99}. 
We then show that all three of them hold in our set-up. Assumption 1.4 in \cite{M99} corresponds to our Assumption \ref{assumption1} and Assumption 1.5 in \cite{M99} follows from the ergodicity of the supOU process. Showing Assumption 1.6 in \cite{M99} means to prove that a stochastic Lipshitz-type assumption holds for each component of the function $h(X_t,\theta)$. 
Let $\theta_i=(\mu_i,\sigma^2_i,\alpha_{\pi}^i,B_i)$ be parameter vectors belonging to $\Theta$ for $i=1,2$. Then, for example for the first component 
\[
\Big| h_{\E}(X_t^{(m)},\theta_1)-h_{\E}(X_t^{(m)},\theta_2)\Big|=\Big| \frac{\mu_1}{B_1(\alpha_{\pi}^1-1)}- \frac{\mu_2}{B_2(\alpha_{\pi}^2-1)} \Big|.    
\] 
That means, by construction, the terms where $X_t$ appears cancel out and the stochastic Lipschitz-type condition reduces to a Lipschitz continuity condition on the non-random terms in each component of $h(X_t,\theta)$. Since the terms have bounded first partial derivatives, they are Lipshchitz continuous and Assumption 1.6 in \cite{M99} holds.
Assumption 1.3 in \cite{M99} is implied by our Assumption \ref{assumption2}. Assumption 1.7 in \cite{M99} is fulfilled by construction, then $h(X_t,\theta)$ is continuously differentiable w.r.t. $\theta$ in $\Theta$. Let $G_{N,m}(X,\theta)=\frac{1}{N-m} \sum_{t=1}^{N-m} \frac{\partial h(X_t,\theta)}{\partial \theta^{\prime}}$, Assumption 1.8 in \cite{M99} requires that a weak law of large numbers applies to $\frac{\partial h(X_t,\theta)}{\partial \theta^{\prime}}$ in a neighborhood of $\theta_0$. That is, for each sequence $\theta^*_N$ such that $ \theta^*_N \xrightarrow{p} \theta_0$ then $G_{N,m}(X,\theta^*_N) \to G_0$.
We have that the matrix
\[
\frac{\partial h(X_t,\theta)}{\partial \theta^{\prime}}=
\]
\begin{equation*}
\footnotesize
\left[\begin{array}{cccc}
\frac{1}{B(\alpha_{\pi}-1)} & 0 & -\frac{\mu}{B(\alpha_{\pi}-1)^2}& -\frac{\mu}{B^2(\alpha_{\pi}-1)}\\
-\frac{2\mu}{B^2(\alpha_{\pi}-1)^2} & \frac{1}{2B(\alpha_{\pi}-1)} &\frac{2\mu^2}{B^2(\alpha_{\pi}-1)^3}-\frac{\sigma^2}{2B(\alpha_{\pi}-1)^2}& \frac{2\mu^2}{B^3(\alpha_{\pi}-1)^2}-\frac{\sigma^2}{2B^2(\alpha_{\pi}-1)}\\
\ldots & \ldots & \ldots & \ldots\\
-\frac{2\mu}{B^2(\alpha_{\pi}-1)^2}&\frac{(1-B\Delta k)^{1-\alpha_{\pi}}}{2B(\alpha_{\pi}-1)}& \tilde{a}(\Delta,\mu,\sigma^2,\alpha_{\pi},B,k)& \tilde{b}(\Delta,\mu,\sigma^2,\alpha_{\pi},B,k)\\
\ldots & \ldots & \ldots & \ldots\\
-\frac{2\mu}{B^2(\alpha_{\pi}-1)^2}&\frac{(1-B\Delta m)^{1-\alpha_{\pi}}}{2B(\alpha_{\pi}-1)}& \tilde{a}(\Delta,\mu,\sigma^2,\alpha_{\pi},B,m)& \tilde{b}(\Delta,\mu,\sigma^2,\alpha_{\pi},B,m)
\end{array}\right],
\end{equation*}
where 
\begin{equation*}
 \tilde{a}(\Delta,\mu,\sigma^2,\alpha_{\pi},B,k)=\frac{2\mu^2}{B^2(\alpha_{\pi}-1)^3}-\frac{\sigma^2((\alpha_{\pi}-1)ln(1-B\Delta k)+1)}{2B(\alpha_{\pi}-1)^2(1-B\Delta k)^{\alpha_{\pi}-1}}
 \end{equation*}
  and
\begin{equation*}
\tilde{b}(\Delta,\mu,\sigma^2,\alpha_{\pi},B,k)= \frac{2\mu^2}{B^3(\alpha_{\pi}-1)^2}+ \sigma^2\frac{B \Delta k (\alpha_{\pi}-1)-(1-B \Delta k)}{2B^2(\alpha_{\pi}-1)(1-B \Delta k)^{\alpha_{\pi}}}
\end{equation*}
for $k=1\ldots,m$.
Then, $\frac{\partial h(X_t,\theta)}{\partial \theta^{\prime}}$ does not depend on $X_t$, $G_{N,m}(X,\theta)=\frac{\partial h(X_t,\theta)}{\partial \theta^{\prime}}$ and $G_0= \frac{\partial h(X_t,\theta_0)}{\partial \theta^{\prime}}$. By the continuous mapping theorem Assumption 1.8 in \cite{M99} then follows.
Assumption 1.9 in \cite{M99} follows by Theorem \ref{asy_mom1} or Corollary \ref{asy_mom2}. Then, because of Assumption \ref{assumption_pos1}, the asymptotic normality of the estimator follows from the same steps as in the proof of \cite[Theorem 1.2]{M99} when we replace $f_T$ and $F_T$ by $g_{N,m}$ and $G_{N,m}$.
\end{proof}

\subsection{SupOU SV model}
We work now with a sample $\{Y_t: t=1,\ldots,N\}$ of the return process and define $Y_t^{(m)}=(Y_{t+1}, Y_{t+2},\ldots,Y_{t+m+1})$ for $t=1,\ldots,N-m$.

The moment function is given by the measurable function
$\tilde{h}: \R^{m+1} \times \Theta \to \R^{m+2}$ as
\begin{equation*}
\tilde{h}(Y_t,\theta)=\left(\begin{array}{l}
\tilde{h}_{Var}(Y_t^{(m)},\theta)\\
\tilde{h}_{0}(Y_t^{(m)},\theta)\\
\tilde{h}_1(Y_t^{(m)},\theta)\\
\,\,\,\,\,\,\,\,\vdots\\
\tilde{h}_m(Y_t^{(m)},\theta)
\end{array}\right)
\end{equation*}
\begin{equation}
\label{mom_func2}
=\left(\begin{array}{l}
Y_{t+1}^2+\frac{\mu \Delta}{B(\alpha_{\pi}-1)}\\
Y_{t+1}^4-3\Big(\frac{\Delta\mu}{B(\alpha_{\pi}-1)}\Big)^2+3\sigma^2
\frac{(1-B\Delta)^{3-\alpha_{\pi}}-1-\Delta B(\alpha_{\pi}-3)}{B^3(\alpha_{\pi}-1)(\alpha_{\pi}-2)(\alpha_{\pi}-3)}\\
Y_{t+1}^2Y_{t+2}^2-\Big(\frac{\Delta\mu}{B(\alpha_{\pi}-1)}\Big)^2+
\sigma^2 \frac{f_{2}-2f_{1}+f_{0}}{2B^3(\alpha_{\pi}-1)(\alpha_{\pi}-2)(\alpha_{\pi}-3)}\\
\,\,\,\,\,\,\,\,\vdots\\
Y_{t+1}^2Y_{t+1+m}^2-\Big(\frac{\Delta\mu}{B(\alpha_{\pi}-1)}\Big)^2+
\sigma^2 \frac{f_{m+1}-2f_{m}+f_{m-1}}{2B^3(\alpha_{\pi}-1)(\alpha_{\pi}-2)(\alpha_{\pi}-3)}
\end{array}\right).
\end{equation}

In this case, the sample moment function of the return process is
\begin{equation}
g_{N,m}(Y,\theta)=\left(\begin{array}{l}
\frac{1}{N-m} \sum_{t=1}^{N-m} \tilde{h}_{Var}(Y_t^{(m)},\theta)\\
\frac{1}{N-m} \sum_{t=1}^{N-m} \tilde{h}_{0}(Y_t^{(m)},\theta)\\
\frac{1}{N-m} \sum_{t=1}^{N-m} \tilde{h}_1(Y_t^{(m)},\theta)\\
\,\,\,\,\,\,\,\,\vdots\\
\frac{1}{N-m} \sum_{t=1}^{N-m} \tilde{h}_m(Y_t^{(m)},\theta)
\end{array}\right)
\end{equation}
\[
=\left(\begin{array}{l}
\frac{1}{N-m} \sum_{t=1}^{N-m} \Big(Y_{t+1}^2+\frac{\mu \Delta}{B(\alpha_{\pi}-1)}\Big) \\
\frac{1}{N-m} \sum_{t=1}^{N-m} 
\Big(Y_{t+1}^4-3\Big(\frac{\Delta\mu}{B(\alpha_{\pi}-1)}\Big)^2+3\sigma^2\frac{(1-B\Delta)^{3-\alpha_{\pi}}-1-\Delta B(\alpha_{\pi}-3)}{B^3(\alpha_{\pi}-1)(\alpha_{\pi}-2)(\alpha_{\pi}-3)}\Big)\\
\frac{1}{N-m} \sum_{t=1}^{N-m} 
\Big(Y_{t+1}^2Y_{t+2}^2-\Big(\frac{\Delta\mu}{B(\alpha_{\pi}-1)}\Big)^2+
\sigma^2 \frac{f_{2}-2f_{1}+f_{0}}{2B^3(\alpha_{\pi}-1)(\alpha_{\pi}-2)(\alpha_{\pi}-3)}\Big)\\
\,\,\,\,\,\,\,\,\vdots\\
\frac{1}{N-m} \sum_{t=1}^{N-m} 
\Big(Y_{t+1}^2Y_{t+1+m}^2-\Big(\frac{\Delta\mu}{B(\alpha_{\pi}-1)}\Big)^2+
\sigma^2 \frac{f_{m+1}-2f_{m}+f_{m-1}}{2B^3(\alpha_{\pi}-1)(\alpha_{\pi}-2)(\alpha_{\pi}-3)} \Big)
\end{array}\right),
\]
and $\theta_0$ can be estimated by minimizing the objective function
\begin{equation}
\label{estimator2}
\hat{\theta}_{0}^{*N,m}= \mathrm{argmin} \, g_{N,m}(Y,\theta)^{\prime} A_{N,m} g_{N,m}(Y,\theta)
\end{equation}
where $A_{N,m}$ is a positive definite matrix to weight the $m+2$ different moments collected in 
$g_{N,m}(Y,\theta)$. 

The consistency of the estimator (\ref{estimator2}) is shown in \cite[Theorem 3.2]{STW15}, and as before we need to show that the moment function $\tilde{h}(Y,\theta)$ satisfies a central limit theorem.
\begin{Theorem}
\label{asy_mom3}
Let $\Lambda$ be a real valued L\'evy basis with generating quadruple $(\gamma,0,\nu,\pi)$, Assumptions $\textrm{(H)}$ be satisfied such that $\int_{|x|>1} |x|^{4+\delta} \, \nu(dx) < \infty$, for some $\delta >0$, and let Assumption \ref{assumption0} hold with $\alpha_{\pi}-1 >  (1+\frac{1}{\delta})(\frac{6+2\delta}{\delta})$. 
Let $(Y_t)_{t \in R}$ be the resulting return process of a supOU SV model, then
\[
W{_\Sigma}= \sum_{l \in \Z} Cov(\tilde{h}(Y_0,\theta_0),\tilde{h}(Y_l,\theta_0)
\]
is finite, positive semidefinite and as $N \to \infty$
\[
\sqrt{N} g_{N,m}(Y,\theta_0) \xrightarrow{d} \mathcal{N}(0, W{_\Sigma}).
\]
\end{Theorem}

\begin{proof}
Proceeding as in Theorem \ref{asy_mom1}, it can be shown that $\tilde{h}(Y_t,\theta_0)$ is a $\theta$-weakly dependent process with zero mean, by using Lemma \ref{lem2} and Proposition \ref{her3}.
Given the $\theta$-coefficients (\ref{theta_sup_3}) , we have $\theta_{\tilde{h}}(r)=\mathcal{C}\mathcal{D}^{*\frac{\delta}{3+\delta}}\Big(-\Delta \frac{\mu (1-B\Delta(r-m-1))^{1-\alpha_{\pi}}}{B(\alpha_{\pi}-1)}\Big)^{\frac{\delta}{6+2\delta}}$, where $\alpha_{\pi}-1 > (1+\frac{1}{\delta})(\frac{6+2\delta}{\delta})$ by assumption.
Then applying \cite[Theorem 1]{DR00} and the Cramer-Wold device the result follows.
\end{proof}

\begin{table}
\footnotesize
\centering
\begin{tabular}{|c|c|}
\hline
$(i,j,k)$ & $A(i,j,k)$\\
\hline
$\{i \neq j \neq k \}$ & $0$\\
\hline
 $\{i\neq j\} \wedge \{j=k\}$ & $ 4 \int_{i\Delta}^{(i+1)\Delta} \int_{j \Delta}^{(j+1)\Delta} \int_{j \Delta}^s \E[X_t X_s X_u] du ds dt $\\
\hline
$\{i=j=k\}$ & $ 12 \int_{i\Delta}^{(i+1)\Delta} \int_{i \Delta}^{(i+1)\Delta} \int_{i \Delta}^s \E[X_t X_s X_u] du ds dt $\\
\hline
\end{tabular}
\caption{Explicit closed formula for the summand $A(i,j,k)$ for $(i,j,k) \in \Z^3$.}
\label{tab3mom}
\end{table}

\begin{sidewaystable}
\centering
\begin{tabular}{|c|}
\hline
$\{l\geq 1\}$ \\
$Cov(Y_1^2,Y_{l+1}^2)=\frac{-\sigma^2(f_{l+1}-2f_l+f_{l-1})}{2B^3(\alpha_{\pi}-1)(\alpha_{\pi}-2)(\alpha_{\pi}-3)}$\\
\hline
$\{l\geq p\}\wedge \{2l-2p+q \geq 2\}$  \\
$K(p,l,l+q)=\frac{-s \sigma^3(2g_{2l-2p+q-1}-g_{2l-2p+q-2}+g_{2l-2p+q+2}-2g_{2l-2p+q+1})}{6B^4(\alpha_{\pi}-1)(\alpha_{\pi}-2)(\alpha_{\pi}-3)(\alpha_{\pi}-4)}$ \\
\hline
$\{l\geq 0\} \wedge \{2l+q\geq 2\}$ \\
$K(0,l,l+q)=\frac{-s \sigma^3(2g_{2l+q-1}-g_{2l+q-2}+g_{2l+q+2}-2g_{2l+q+1})}{6B^4(\alpha_{\pi}-1)(\alpha_{\pi}-2)(\alpha_{\pi}-3)(\alpha_{\pi}-4)}$ \\
\hline
$\{l\geq 0\}  \wedge \{2l+q+p \geq 3 \}$  \\
$K(0,p,l,l+q)=-\frac{(\eta-3)\sigma^4(-2h_{2l+q+p}+3h_{2l+q+p-1}-3h_{2l+q+p-2}+h_{2l+q+p-3}+h_{2l+q+p+3}-3h_{2l+q+p+2}+3h_{2l+q+p+1})}{12B^5(\alpha_{\pi}-1)(\alpha_{\pi}-2)(\alpha_{\pi}-3)(\alpha_{\pi}-4)(\alpha_{\pi}-5)}$\\
\hline
\end{tabular}
\caption{Explicit closed formula for the summands (\ref{11sv}),(\ref{1psv}) and (\ref{pqsv}), where $p \geq q$, $f_k=(1-B\Delta k)^{3-\alpha_{\pi}}$, $g_k=(1-B\Delta k)^{4-\alpha_{\pi}}$, $h_k=(1-B\Delta k)^{5-\alpha_{\pi}}$ and $\mu=\E[L_1]$, $\sigma^2= Var[L_1]$, $s=\sigma^{-3}\E[(L_1-\mu)^3]$ and $\eta= \sigma^{-4}\E[(L_1-\mu)^4]$.}
\label{tab2}
\end{sidewaystable}

\begin{Remark}

We observe that,
\begin{align*}
&\E[Y_1^2]=\E[V_1]=-\frac{\Delta \mu}{B(\alpha-1)}:=C^*,\\
&Var[Y_1^2]= 3 Var(V_1)+2\E(V_1)^2=-3 \sigma^2 \frac{(1-B \Delta)^{3-\alpha_{\pi}}-1-\Delta B(\alpha_{\pi}-3)}{B^3(\alpha_{\pi}-1)(\alpha_{\pi}-2)(\alpha_{\pi}-3)}
\end{align*}
\[ +2 \Big(-\frac{\Delta \mu}{B (\alpha_{\pi}-1)} \Big)^2 :=D^*(0),
\]
and
\[
Cov(Y_1^2,Y_{1+k}^2)=Cov(V_1,V_{1+k})=-\sigma^2 \frac{f_{k+1}-2f_{k}+f_{k-1}}{2B^3(\alpha-1)(\alpha-2)(\alpha-3)}:=D^*(k),
\]
where $V$ is the integrated process as defined in (\ref{integrated}) and $C^*$ and $D^*(k)$ are defined in Remark \ref{cum}.
Thus, the coefficients of the matrix $Cov(\tilde{h}(Y_0,\theta_0),\tilde{h}(Y_l,\theta_0)$ for $l \in \Z$ and $p,q \in \{0,\ldots,m\}$ are
\begin{equation}
\label{11sv}
Cov(\tilde{h}_{Var}(Y_0^{(m)},\theta_0),\tilde{h}_{Var}(Y_l^{(m)},\theta_0)= D^*(l),
\end{equation}
\begin{align}
\label{1psv}
\begin{split}
Cov(\tilde{h}_{Var}(Y_0^{(m)},\theta_0),\tilde{h}_{p}(Y_l^{(m)},\theta_0))=&K(0,l,l+p)+A(0,l,l+p)\\
+&C^* (D^*(l)+D^*(l+p)),
\end{split}
\end{align}
\begin{align}
\label{pqsv}
\begin{split}
&Cov(\tilde{h}_{p}(Y_0^{(m)},\theta_0),\tilde{h}_{q}(Y_l^{(m)},\theta_0))=K(0,p,l,l+q)+A(0,p,l,l+q)\\
&+C^*(K(0,p,l)+K(0,p,l+q)+K(p,l,l+q)+K(0,l,l+q))\\
&+C^{*2}(D^*(l+q-p)+D^*(l-p)+D^*(l+q)+D^*(l))\\
&+D^*(l)D^*(l+q-p)+D^*(l+q)D^*(l-p),
\end{split}
\end{align}
where $A(i,j,k)$ and $A(i,j,k,l)$ are defined in Table \ref{tab3mom} and Table \ref{tab4mom}, respectively, and $K(i,j,k)$ and $K(i,j,k,l)$ in (\ref{cumulant3_int}) and (\ref{cumulant_int}). 
In Table \ref{tab2}, the explicit expressions of the summands in (\ref{11sv}), (\ref{1psv}) and (\ref{pqsv}) can be found for a selection of indices $l,p,q$. 
The remaining cases can be easily derived by using the calculations in Table \ref{tab1}.
\end{Remark}

\begin{Remark}
Note that if all the moments of the underlying L\'evy process exist than the asymptotic result of Theorem \ref{asy_mom3} holds assuming that $\alpha_{\pi}>3$. 
\end{Remark}

Additional assumptions have to be made before showing the asymptotic normality of the GMM estimator (\ref{estimator2}):
\begin{Assumption}
\label{assumption4}
The  parameter vector $\theta_0$ is identifiable, i.e. $\E[\tilde{h}(Y,\theta)]=0$ for all $Y$ if and only if $\theta=\theta_0$. 
\end{Assumption}
\begin{Assumption}
\label{assumption_pos2}
The matrix $W_{\Sigma}$ is positive definite.
\end{Assumption}

Regarding Assumption \ref{assumption4}, it has been shown in \cite[Corollary 3.6]{STW15} that identifiability holds if the number of lags $m$ in the definition of the moment function is infinity, the so called asymptotic identifiability. In practice, we always work with a finite number of lags. Although proving identifiability rigorously seems to be out of reach, the asymptotic identifiability suggests that Assumption \ref{assumption4} should be satisfied if $m$ is sufficiently large. It will be interesting to analyze how this affects the precision of our estimator in a simulation study. This is, however, beyond the scope of the present paper.

Finally, we recall that in our set-up $\mu> 0$, $\sigma^2 > 0$, $\alpha_{\pi}>2$ and $B<0$ and the parameter space $\Theta$ is large enough to contain the true parameter vector. 

\begin{Theorem}
\label{GMM_sv}
Let $\Lambda$ be a real valued L\'evy basis with generating quadruple $(\gamma,0,\nu,\pi)$, Assumptions $\textrm{(H)}$ be satisfied such that $\int_{|x|>1} |x|^{4+\delta} \, \nu(dx) < \infty$, for some $\delta >0$, and
let Assumption \ref{assumption0} hold with $\alpha_{\pi}-1 >  (1+\frac{1}{\delta})(\frac{6+2\delta}{\delta})$. 
If, moreover, Assumptions \ref{assumption1}, \ref{assumption2}, \ref{assumption4} and \ref{assumption_pos2} hold, then as $N$ goes to infinity
\[
\sqrt{N}(\hat{\theta}_{0}^{*N,m}-\theta_0) \xrightarrow{d} \mathcal{N}(0, M W_{\Sigma} M^{\prime})
\]
where  
\begin{align*}
M= \E[G_0^{*\prime} A G_0^*]^{-1}G_0^{*\prime} A,&\,\,\, G_0^*=\E[\frac{\partial \tilde{h}(Y_t,\theta)}{\partial \theta^{\prime}}]_{\theta=\theta_0},\\
\textrm{and}\,\,\, W{_\Sigma}= \sum_{l \in \Z} Cov&(\tilde{h}(Y_0,\theta_0),\tilde{h}(Y_l,\theta_0).
\end{align*}

\end{Theorem}

\begin{proof}
We check that Assumptions 1.1-1.9 in \cite{M99} hold.
Assumptions 1.1-1.7 follow by Assumptions \ref{assumption1}, \ref{assumption2} and \ref{assumption4} and by construction of the moment function. The line of the proof, in this case, is exactly the same as in Theorem \ref{GMM_asy}.
The matrix
\[
\frac{\partial \tilde{h}(Y_t,\theta)}{\partial \theta^{\prime}}
\]
is equal to
\[
\left[\begin{array}{cccc}
\frac{\Delta}{B(\alpha_{\pi}-1)} &0 & -\frac{\Delta \mu}{B(\alpha_{\pi}-1)^2}&-\frac{\Delta \mu}{B^2(\alpha_{\pi}-1)}\\
-\frac{6\Delta^2\mu}{B^2(\alpha_{\pi}-1)^2}& \frac{3f_1-1-\Delta B(\alpha_{\pi}-3)}{B^3(\alpha_{\pi}-1)(\alpha_{\pi}-2)(\alpha_{\pi}-3} & a(\Delta,\mu,\sigma^2,\alpha_{\pi},B) & b(\Delta,\mu,\sigma^2,\alpha_{\pi},B)\\
 \ldots& \ldots &\ldots & \ldots\\
-\frac{2\Delta^2\mu}{B^2(\alpha_{\pi}-1)^2}& \frac{f_{m+1}-2f_m-f_{m-1}}{2B^3(\alpha_{\pi}-1)(\alpha_{\pi}-2)(\alpha_{\pi}-3)}& c(\Delta,\mu,\sigma^2,\alpha_{\pi},B,m)& d(\Delta,\mu,\sigma^2,\alpha_{\pi},B,m) \end{array}\right]
\]
where 
{\footnotesize
\begin{align*}
a(\Delta,\mu,\sigma^2,\alpha_{\pi},B)=& \frac{6\Delta^2\mu^2}{B^2(\alpha_{\pi}-1)^3}+3\frac{\sigma^2}{B^3} \frac{(\alpha_{\pi}-2)(\alpha_{\pi}-3)+(\alpha_{\pi}-1)(\alpha_{\pi}-2)+(\alpha_{\pi}-1)(\alpha_{\pi}-2)}{(\alpha_{\pi}-1)^2(\alpha_{\pi}-2)^2(\alpha_{\pi}-3)^2}\\
&-\frac{3\sigma^2}{B^3} f_1 l_1+\frac{3\sigma^2\Delta ((\alpha_{\pi}-1)+(\alpha_{\pi}-2))}{B^2(\alpha_{\pi}-1)^2(\alpha_{\pi}-2)^2},
\end{align*}
\begin{align*}
&b(\Delta,\mu,\sigma^2,\alpha_{\pi},B)=\frac{6\Delta^2\mu^2}{B^3(\alpha_{\pi}-1)^2}+\frac{3\sigma^2}{(\alpha_{\pi}-1)(\alpha_{\pi}-2)(\alpha_{\pi}-3)} \Big(\frac{(\alpha_{\pi}-3)(2\Delta+\Delta r_1)}{B^3}-3\frac{f_1-1}{B^4} \Big),\\
&c(\Delta,\mu,\sigma^2,\alpha_{\pi},B,k)=\frac{2\Delta^2\mu^2}{B^2(\alpha_{\pi}-1)^3}-\frac{\sigma^2}{2B^3}(f_{k+1}l_{k+1} - 2 f_{k}l_k+ f_{k-1}l_{k-1} )\,\, \textrm{for $k=1,\ldots,m$},
\end{align*}}
and
{\footnotesize
\begin{align*}
d(\Delta,\mu,\sigma^2,\alpha_{\pi},B,k)=& \frac{2\Delta^2\mu^2}{B^3(\alpha_{\pi}-1)^2}+\frac{\sigma^2}{2B^3(\alpha_{\pi}-1)(\alpha_{\pi}-2)}(r_{k+1} \Delta (k+1)-2 r_k \Delta k\\
& + r_{k-1}\Delta (k-1))-\frac{3\sigma^2}{2 B^4 (\alpha_{\pi}-1)(\alpha_{\pi}-2)(\alpha_{\pi}-3)} (f_{k+1}-2f_k+f_{k-1}),  
\end{align*}}
for $k=1,\ldots,m$, with $r_k:=(1-B\Delta k)^{2-\alpha_{\pi}}$, $f_k:=(1-B\Delta k)^{3-\alpha_{\pi}}$ and
{\scriptsize
\begin{align*}
&l_k=\frac{\ln(1-B\Delta k)(\alpha_{\pi}-1)(\alpha_{\pi}-2)(\alpha_{\pi}-3)+(\alpha_{\pi}-2)(\alpha_{\pi}-3)+(\alpha_{\pi}-1)(\alpha_{\pi}-3)+(\alpha_{\pi}-1)(\alpha_{\pi}-2)}{(\alpha_{\pi}-1)^2(\alpha_{\pi}-2)^2(\alpha_{\pi}-3)^2)},
\end{align*}}

for $k \in \{0,\ldots,m+1\}$. 
Therefore, $ \frac{\partial \tilde{h}(Y_t,\theta)}{\partial \theta^{\prime}}$ does not depend on $Y_t$ and, as in the proof of Theorem \ref{GMM_asy}, the continuous mapping theorem can be applied to show that Assumption 1.8 in \cite{M99} holds. Assumption 1.9 follows by the proof of Theorem \ref{asy_mom3}.
Then, because of Assumption \ref{assumption_pos2}, the normality of the GMM estimator follows from the same steps as in the proof of \cite[Theorem 1.2]{M99}.
\end{proof}

In \cite[Section 4]{STW15} a simulation study on the estimators (\ref{estimator}) and (\ref{estimator2}) looks at their finite sample performances (Theorem \ref{GMM_asy} and \ref{GMM_sv} are applicable to the set-up of the study). The analysis performed shows results in line with  asymptotic normality derived theoretically in this paper. To obtain reliable estimation of supOU processes and supOU SV model, a substantial amount of data is needed. 

\section*{Acknowledgements}
We thank the Deutsche Forschungsgemeinschaft for the financial support through the research grant STE 2005/1-1. Finally, we would like to thank the two anonymous reviewers and the editor for their insightful comments.

\section{Erratum: Weak dependence and GMM estimation of supOU and mixed moving average processes}

\setcounter{section}{3}
\setcounter{Proposition}{1}
\setcounter{equation}{9}
This erratum concerns the proof of Proposition \ref{her2}. Here, it is erroneously stated that the truncated process $X_i^{(M)}:= X_i \mathbb{1}_{\{\|X_{i}\|\leq M\}} $ is $\eta$-weakly dependent with the same coefficients as the process $X_i$. We need to redefine $X_{i}^{(M)}=P_M(X_i)$, where $P_M$ indicates the orthogonal projection of $X_i$ into the closed ball $\{x \in \R^m: \|x\| \leq M\}$. With this choice, the truncated process is $\eta$-weakly dependent, and the proof of Proposition \ref{her2} can be correctly carried out. For completeness, we report anew  the proof of Proposition \ref{her2} below.

\begin{Proposition}
Let $(X_t)_{t \in \R}$ be an $\R^n$-valued stationary process and assume there exists some constant $C>0$ such that $ \E[\|X_{0}\|^p]^{\frac{1}{p}} \leq C,$ with $p>1$, $h \colon \R^n \to \R^m$ be a function such that $h(0)=0$, $h(x)=(h_1(x),\ldots,h_m(x))$ and
	\begin{equation}
		\|h(x)-h(y)\| \leq \tilde{c} \,\|x-y\| (1+\|x\|^{a-1}+\|y\|^{a-1}),
	\end{equation}
	for $x,y \in \R^n$, $\tilde{c}>0$ and $1\leq a < p$. 
	Define $(Y_t)_{t \in \R}$ by $Y_t=h(X_t)$. If  $(X_t)_{t \in \R}$ is an $\eta$-weakly dependent process,
	then $(Y_t)_{t \in \R}$ is an $\eta$-weakly dependent process such that
	\[
	\forall \, r \geq 0, \,\,\, \eta_Y(r)= \mathcal{C} \, \eta_X(r)^{\frac{p-a}{p-1}},
	\]
	with the constant $\mathcal{C}$ independent of $r$.

\end{Proposition}

\begin{proof}
	
	For $(u,v) \in \N^*\times \N^*$, $(i_1,\ldots,i_u) \in \R^u$ and $(j_1,\ldots,j_v) \in \R^v$ with $i_1\leq\ldots\leq i_u \leq i_u+r\leq j_1\leq \ldots\leq j_v$, let us call
	\[
	X_i^*=(X_{i_1},\ldots,X_{i_u}),\,\,X_{j}^*=(X_{j_1},\ldots,X_{j_v}).
	\]
	
	Let $F:\R^{mu}\to \R$, $G:\R^{mv} \to \R$ be bounded Lipschitz functions, such that $\|F\|_{\infty},\|G\|_{\infty} \leq 1$, and define
	\begin{align*}
		\tilde{F}(X_i^*)=F(h(X_{i_1}),\ldots,h(X_{i_u})), \,\, \tilde{G}(X_j^*)= G(h(X_{j_1}),\ldots,h(X_{j_v})),
	\end{align*}
	and
	\[
	\tilde{F}^{(M)}(X_i^*)= F(h(X_{i_1}^{(M)}),\ldots,h(X_{i_u}^{(M)})),\,\, \tilde{G}^{(M)}(X_j^*)=G(h(X_{j_1}^{(M)}),\ldots,h(X_{j_v}^{(M)})),
	\]
	for $X_{i}^{(M)}=P_M(X_i)$, where $P_M$ indicates the orthogonal projection of $X_i$ into the closed ball $\{x \in \R^m: \|x\| \leq M\}$, and w.l.o.g $M>1$.
	We start by analyzing 
	\begin{align}
		\label{bound}
		\begin{split}
			&|Cov(\tilde{F}(X_i^*),\tilde{G}(X_j^*))| \leq |Cov(\tilde{F}(X_i^*)-\tilde{F}^{(M)}(X_i^*),\tilde{G}(X_j^*))|\\
			&+|Cov(\tilde{F}^{(M)}(X_i^*),\tilde{G}(X_j^*)-\tilde{G}^{(M)}(X_j^*))|+|Cov(\tilde{F}^{(M)}(X_i^*), \tilde{G}^{(M)}(X_j^*))|.
		\end{split}
	\end{align}
	We have that
	\begin{equation}
		\label{est1}
		|Cov(\tilde{F}(X_i^*)-\tilde{F}^{(M)}(X_i^*),\tilde{G}(X_j^*))|\leq 2 \|G\|_{\infty} \E|\tilde{F}(X_i^*)-\tilde{F}^{(M)}(X_i^*)|
	\end{equation}
	\[
	\leq 2 Lip(F) \sum_{l=1}^u \E \|h(X_{i_l})-h(X_{i_l}^{(M)})\|.
	\]
	By assumption, for each $l=1,\ldots,u$
	\begin{equation}
		\label{conto1}
		\E (\|h(X_{i_l})-h(X_{i_l}^{(M)})\|) \leq \tilde{c} \,\E (\|X_{i_l}-X_{i_l}^{(M)}\| ( 1+\|X_{i_l}\|^{a-1} + \|X_{i_l}^{(M)}\|^{a-1}))
	\end{equation}
	\[
	\leq \tilde{c} \, \E (\| X_{i_l} \| \mathbb{1}_{\|X_{i_l}\| >M})+ \tilde{c} \,\E (\|X_{i_l}\|^{a-1} \| X_{i_l} \| \mathbb{1}_{\|X_{i_l}\| >M})
	\]
	\[
	+\tilde{c} \,\E (\| X_{i_l} \| \mathbb{1}_{\|X_{i_l}\| >M}) M^{a-1}
	\]
	\[
	\leq \tilde{c} \, \E \Big(\| X_{i_l} \| \frac{\| X_{i_l} \|^{p-1}}{M^{p-1}} \mathbb{1}_{\|X_{i_l}\| >M} \Big)+ \tilde{c} \,\E \Big(\|X_{i_l}\|^{a-1} \frac{\| X_{i_l} \|^{p-a} }{M^{p-a}} \| X_{i_l} \| \mathbb{1}_{\|X_{i_l}\| >M}\Big)
	\]
	\[
	+ \tilde{c} \,\E \Big(\| X_{i_l} \| \frac{\| X_{i_l} \|^{p-1}}{M^{p-1}} \mathbb{1}_{\|X_{i_l}\| >M}\Big) M^{a-1}
	\]
	\[
	\leq \tilde{c} \, \E( \| X_{i_l} \|^p \mathbb{1}_{\|X_{i_l}\| >M}) M^{1-p} + 2 \tilde{c} \, \E( \| X_{i_l} \|^p \mathbb{1}_{\|X_{i_l}\| >M}) M^{a-p}
	\]
	\[
	\leq 3\tilde{c}\, \E (\|X_{i_l}\|^p)  M^{a-p}.
	\]
	
	Therefore, (\ref{est1}) is less than or equal to $6 \tilde{c} \,u\,Lip(F) C^p M^{a-p}$. An analogous bound holds for $|Cov(\tilde{F}^{(M)}(X_i^*),\tilde{G}(X_j^*)-\tilde{G}^{(M)}(X_j^*))|$.
	Moreover, $\tilde{F}^{(M)}$ is a Lipschitz function on $\R^{mu}$.  Let $Z=(Z_1,\ldots,Z_u)$ and $W=(W_1,\ldots,W_u)$ where $Z_i, W_i \in \R^{m}$ for $i=1,\ldots,u$, then 
	\begin{align*}
		&|\tilde{F}^{(M)}(Z)-\tilde{F}^{(M)}(W)| =|F(h(Z^{(M)}_1),\ldots,h(Z^{(M)}_u))-F(h(W^{(M)}_1), \ldots, h(W^{(M)}_u))|\\
		&\leq Lip(F) \sum_{l=1}^u \|h(Z_{l}^{(M)})-h(W_{l}^{(M)})\| \\
		&\leq \tilde{c} \, Lip(F) \sum_{l=1}^u  \|Z_{l}^{(M)} -W_{l}^{(M)}\|(1+ \| Z_{l}^{(M)}\|^{a-1} \|W_l^{(M)}\|^{a-1})\\
		&\leq 3\tilde{c} \, Lip(F) M^{a-1} \sum_{l=1}^u  \|Z_{l}^{(M)} -W_{l}^{(M)}\| \leq 3\tilde{c} \, Lip(F) M^{a-1} \sum_{l=1}^u  \|Z_{l} -W_{l}\|.
	\end{align*}
	The last inequality holds because the orthogonal projection is a Lipschitz function with $Lip(P_M)=1$. Note that this latter property also implies that the process $X_i^{(M)}$ is itself $\eta$-weakly dependent.
	
	A similar argument can be used to prove that also the function $G^{(M)}$ is Lipschitz on $\R^{mu}$. Hence, because $X_t$ is $\eta$-weakly dependent,
	\begin{align*}
		|Cov(\tilde{F}^{(M)}(X_i^*), \tilde{G}^{(M)}(X_j^*))|
		&\leq  c \,( u Lip(\tilde{F}^{(M)})  +  v Lip(\tilde{G}^{(M)}) ) \eta_X(r)    \\
		&= 3 \tilde{c} \, c \, M^{a-1}  \, ( u Lip(F)  +  v Lip(G) )  \eta_X(r)
	\end{align*}
	To conclude, (\ref{bound}) is less than or equal to
	\[
	6 \tilde{c} \,(u Lip(F) + v Lip(G)) \Big( \frac{c M^{a-1}}{2} \eta_X(r) + C^p M^{a-p} \Big).
	\]
	The process $(Y_t)_{t \in \R}$ defined by $Y_t=h(X_t)$ is then $\eta$-weakly dependent. By choosing  $M= \eta_X(r)^{\frac{1}{1-p}}$ and calling $\mathcal{C}=6\tilde{c} (C^p+\frac{c}{2})$, we obtain that
	\[
	\eta_Y(r)= \mathcal{C} \, \eta_X(r)^{\frac{p-a}{p-1}}.
	\]

\end{proof}

\end{document}